\documentclass{amsproc}

\usepackage{amssymb}

\usepackage{graphicx}

\newtheorem{theorem}{Theorem}[section]
\newtheorem{lemma}[theorem]{Lemma}
\newtheorem{corollary}[theorem]{Corollary}

\theoremstyle{definition}

\theoremstyle{remark}
\newtheorem{remark}[theorem]{Remark}

\numberwithin{equation}{section}

\newcommand{\ex}{e^{\frac{iqx_1}{\sqrt{1+b^2}}}}
\newcommand{\n}{\mathbf{n}}

\newcommand{\co}{\cos \theta}
\newcommand{\cg}{\overline{\psi}}
\newcommand{\bs}{\sqrt{1+b^2}}

\begin{document}

 \title[short text for running head]{Dynamic Analysis of Chevron Structures}
\title{Dynamic Analysis of Chevron Structures in Liquid Crystal Cells}


\author{Lidia Mrad}
\address{ Department of Mathematics, University of Arizona}
\curraddr{Tucson, AZ\ 85721}
\email{lmrad@math.arizona.edu}
\thanks{}

\author{Daniel Phillips }
\address{Department of Mathematics, Purdue University}
\curraddr {West Lafayette, IN\ 47907}
\email{phillips@purdue.edu}
\thanks{Research supported by NSF grant DMS-1412840}


\date{}
\begin{abstract}

If a surface stabilized ferroelectric liquid crystal cell is cooled from the smectic-A to the smectic-C
phase its layers thin causing V-shaped (chevron like) defects to form. These  create an energy barrier that can prevent switching between equilibrium patterns.  We examine
a gradient flow for a mesoscopic Chen�-Lubensky energy $\mathcal{F}(\psi,\n)$
 that allows the order parameter  to vanish, so that the energy barrier does not diverge if the layer thickness
becomes small. The liquid crystal can evolve during switching in such a way that
the layers are allowed to melt and heal near the chevron tip
in the process.

\end{abstract}

\maketitle
\section{Introduction}
In surface stabilized ferroelectric liquid crystal [SSFLC] cells
smectic layers usually deform into a characteristic chevron
pattern [1, 2]. The chevron structure (see Fig. \ref{fig:Chevron}) is believed
to arise due to the mismatch between the natural smectic
layer thickness and the periodicity imposed by the layer pinning
at the surface in the smectic-A phase, where this surface
memory effect has been confirmed experimentally [3].
In the past,  several theoretical models have
been presented to describe the director and layer structure in
smectic-C chevron cells. This  work has
been motivated by a potential use of [SSFLC] cells in display devices. The original model was put
forward by Clark and co-workers [1, 4]. Theirs is a macroscopic description where the molecular alignment
varies slowly across each chevron arm, restricted so that the molecules lie on the arm's smectic-C cone. The chevron tip is idealized to have zero thickness; the smectic layers form a sharp bend at the tip so that  their normal is discontinuous
there. A key assumption in their model is the
continuity of the equilibrium director pattern across the cell. Assuming the cone angle is larger than the layer tilt, these conditions lead to two out-of-plane states determined by the intersection of the cones from each side of the tip. This means that chevron cells
exhibit two stable director states  between
which the cell can be switched by the application of an
external electric field. An important reason to model surface
stabilized cells is to describe the switching dynamics between
these stable director states. The model of Clark et al. has
been extended to include continuity of the biaxial ordering at
the chevron tip [5].

\begin{figure}[h]
  \centering
  \includegraphics[width=2.3in]{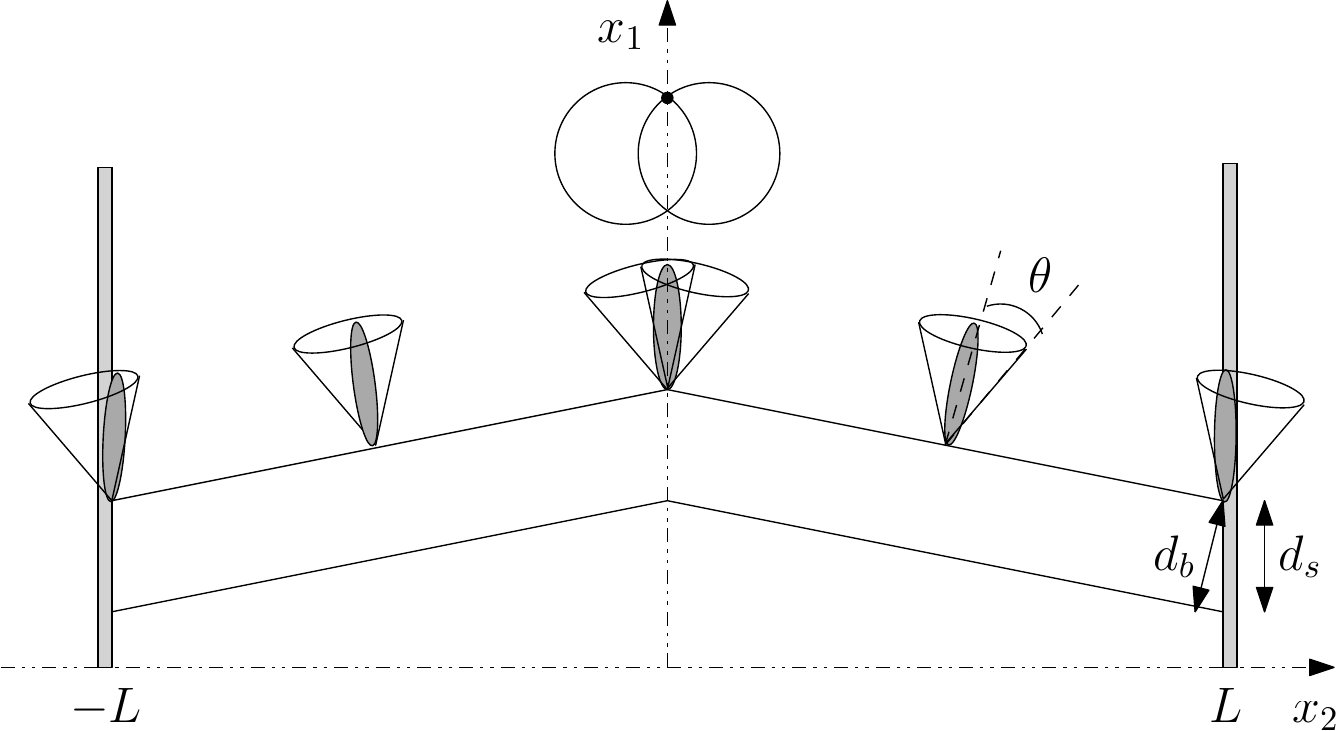}\hspace{0.3in}\includegraphics[width=2.3in]{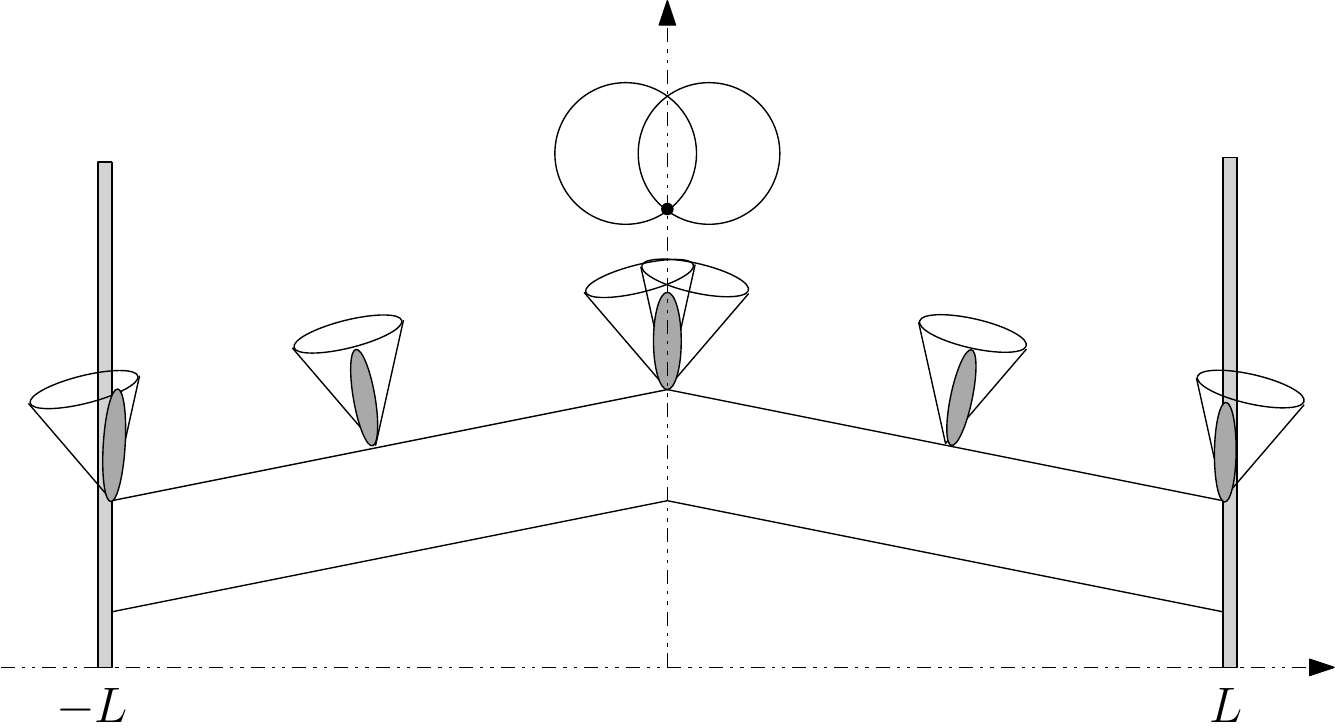}
  \caption{Switching in a Chevron Structure}
  \label{fig:Chevron}
\end{figure}

The development of models to describe ferroelectric
switching has a number of difficulties. These are largely due to the
contrasting length scales associated with the chevron
arms and tip. The macroscopic models such as the one described above require that the director at the chevron tip leaves the cones from the adjacent arms during the switching process. The models account for this by allowing the director to be discontinuous at the tip during the transition and an ad hoc energy barrier is included in the free energy  penalizing  the director's jump across the tip. (See Maclennan
et al. [6, 7], Ulrich and Elston [8], and Brown et al. [9]).

A number of models of the chevron structure have been
reported  which take into account continuous layer bending at the chevron tip and the director
rotation on the cone [10 - 15]. These models
however, enforce specific couplings between the molecular cone angle from the
smectic layer tilt angle that permit a continuous pattern near the tip but are also ad hoc.

In this paper we use a Chen-Lubensky model to characterize the director and layer
structure of surface stabilized smectic-C cells  in an external electric field. This approach has been  developed in the papers Kralj and Sluckin [16], Vaupoti\u{c} et al. [17, 18], Hazelwood and Sluckin [19], and Cheng and Phillips [20]. We build on these works here. The advantage of this phenomenological model over others is that the formation of chevrons and the nature of the director pattern near the chevron tip follow from energetic considerations as opposed to features that are directly inserted into the model. In [20], the smectic layer thickness is introduced as a small parameter and the static model of Clark et al. described above is captured as a singular limit having smectic layers with a sharp chevron tip and its two out-of-plane equilibrium director patterns. In this work we consider a more complete Chen-Lubensky model that allows the smectic layers to melt near the chevron tip under the application of a finite electric field allowing for the switching from one state, tending to the other assuming gradient flow dynamics. Our contribution here is that we establish existence and uniqueness results for the gradient flow problem. Many open questions for this model remain to be explored. For example once the switching occurs the model allows the smectic layers to reform near the chevron tip. Just how robust this regeneration is remains an open question. This is an important point since the studies of  Willis et al. [21] showed
 that there should be no significant change in smectic layer thickness or
chevron layer structure under typical director switching conditions.

\section{Model}

We consider a cross section of an SSFLC cell of width $2L$ and place the origin of the axes at the midpoint of the lower cell aperture. The domain  is then written as $\{(x_1,x_2); -L<x_2<L\}$.

To model a smectic phase, we need two types of ordering, orientational and positional. The former is described by a unit vector  $\n$ that indicates the average local orientation of the rod-shaped molecules, called the nematic director field. The latter is described by a complex-valued order parameter, $\Psi$, whose amplitude indicates the degree of smectic ordering and whose phase indicates the position of the smectic layer. If $|\Psi|=0$, then the phase is completely nematic. For a smectic phase, $|\Psi|>0$, with $|\Psi|\ge1$ for well-structured smectic layers.

An essential parameter in our analysis is the  wave number $q$, which equals $2\pi/d_b$, where $d_b$ is the layer thickness in the bulk smectic-C phase. In a realistic model, $d_b$ is very small (typically $L/d_b$ is in the range $300-400$ [17]) , so one eventually would like to study the limiting problem letting $q\to \infty$. For the present work, $q$ is assumed to be a large number. Let $d_s$ denote the layer thickness at the surface. This is a spacing  inherited from the material initially being cooled from the smectic-A phase. We use the parameter $b=\tan\left(\cos^{-1} \frac{d_b}{d_s}\right)$ to measure the mismatch between the different layer thicknesses.
Lastly, let $\theta$  denote the bulk tilt angle  of the molecules from the layer normal in the smectic-C phase.

We base our model on a covariant form of the Landau-de Gennes free energy, introduced by Chen \& Lubensky [22]. The density for our free energy  $\mathcal{F}(\n, \Psi)$ consists of three parts, nematic, smectic and electrostatic

$$ \mathcal{F}(\n, \Psi)=\int\{f_N(\n)+ \frac{1}{q}f_{CL}(\n,\Psi)+f_E(\n, \Psi)\}\,.$$

The nematic density, $f_N(\n)$, measures the uniformity of $\n$ and we take it to be the one-constant approximation of the Oseen-Frank energy density, $\frac{1}{2}K|\nabla{\mathbf{n}}|^2$ where $K>0$. The smectic part, $\frac{1}{q} f_{CL}(\n,\Psi)$, measures the uniformity of the layer structure and is a variation of the Chen-Lubensky energy density. We split this energy density into two components, $F(\n,\Psi)$ and $G(|\Psi|^2)$. The first component is the elastic energy density,
$$F(\n,\Psi)=\frac{a_{\perp}}{q^3}|D\cdot D_{\perp}\Psi|^2+\frac{a_{\|}}{q^3}|D\cdot D_{\|}\Psi|^2-\frac{c_{\perp}}{q}|D_{\perp}\Psi|^2+\frac{c_{\|}}{q}|D_{\|}\Psi|^2+\frac{c_{\perp}^2q}{4a_{\perp}}|\Psi|^2.$$
The second component we call the smectic penalization density,
$$G(|\Psi|^2)=g(|\Psi|^2-1)^2+|\nabla|\Psi|^2|^2+\frac{1}{q^2}|\nabla^2|\Psi|^2|^2+\frac{1}{q^6}|\nabla^3|\Psi|^2|^2.$$
$D$ is the covariant derivative, $D=\nabla-iq\co\n$ with its parallel component (to $\n$) $D_{\|}=(\n\cdot D)\n$ and perpendicular component (to $\n$) $D_{\perp}$. The  parameter constants $a_{\perp}, a_{\|}, c_{\|}, c_{\perp}, g$ are positive and determined by the material. Lastly, we write the electrostatic part, $f_E(\n, \Psi)=-\mathbf{P}\cdot \mathbf{E}$, where $\mathbf{P}=\frac{P}{q}\Im\{\overline{\Psi}\nabla\Psi\}\times\n$ is the spontaneous polarization field and $\mathbf{E}=(0,E,0)$ is the applied electric field directed across the cell.

Due to the periodicity of $\Psi$, we are able to reduce our model to be one-dimensional. Specifically, we write $\Psi(x_1,x_2)=\ex \psi(x_2)$ where $\psi$ is complex-valued. Expressing periodicity in this form was used by Kralj and Sluckin [16], wherein to enforce the smectic density wave in the $x_1$ direction, $\Psi$ was written in the form $\Psi=\eta e^{iq\frac{[x_1-g(x_2)]}{\sqrt{1+b^2}}}$ such that $\eta>0$ was assumed to be a constant and $x_1=g(x_2)$ is the graph of a uniformly smectic layer in their setting. The same idea was also used in [20]. Our approach, however, is different in that it aims to explain the whole process through the model without invoking any ad hoc energy terms added later. Particularly, we want the system to allow for phase changes in certain areas if that is less costly in an energetic sense (by melting, for instance). The tool to do that is to keep the complex-valued order parameter in a general form, allowing $\psi$ to vanish.

We consider the admissible set
\begin{align*}
X=&\left\{(\n,\psi) \in H^1((-L,L);\mathbb{S}^2)\times H^3((-L,L);\mathbb{C}); \right.&\\
&\qquad\left. \psi-\psi_0\in H^2_0((-L,L);\mathbb{C}) \text{ and }|\psi|^2{}''(-L)=|\psi|^2{}''(L)=0\right\}&
\end{align*}
where $\psi_0 \in H^3((-L,L);\mathbb{C})$ such that $|\psi_0(\pm L)|=1 \text{ and } \psi_0'\cg_0(\pm L)=\pm\frac{iqb}{\bs}.$ 
  The boundary conditions enforce uniform smectic layers at the cell surfaces and for simplicity there are no boundary conditions on the director. We rewrite the smectic energy by integrating by parts and taking into account the boundary conditions to get an  integrand that is bounded below. We then carry out the dimension reduction with our variable denoted by $x$ rather than $x_2$. Within the above admissible set, our total free energy ( per unit length with respect to $x_1$ ) becomes
\begin{flalign}
\mathcal{F}(\psi,\n)=&\int_{-L}^L \left\{\frac{a_{\perp}}{q}|\frac{\psi''}{q}-[(\frac{i}{\bs}n_1\psi+n_2\frac{\psi'}{q})n_2]'-\frac{q}{1+b^2}\psi +\frac{q}{1+b^2}n_1^2\psi\right.&\nonumber\\
&\qquad\qquad\quad-\frac{i}{\bs}n_1n_2\psi'+\frac{c_{\perp}q}{2a_{\perp}}\psi|^2&\nonumber\\
&\qquad+\frac{a_{\|}}{q}|(\frac{n_1}{\bs}-\co)(-\frac{q}{\bs}n_1 \psi +in_2\psi'+q\co \psi)&\nonumber\\
&\qquad\qquad\quad\label{Energy}+[(\frac{i}{\bs}n_1\psi+n_2\frac{\psi'}{{q}}-i\co \psi)n_2]'|^2&\\
&\qquad+qc_{\|}|\frac{i}{\bs}n_1 \psi+n_2 \frac{\psi'}{q}-i\co \psi|^2&\nonumber\\
&\qquad+g(|\psi|^2-1)^2+(|\psi|^2{}')^2+\frac{1}{q^2}(|\psi|^2{}'')^2+\frac{1}{q^6}(|\psi|^2{}''')^2+\frac{\rho}{q^6}|\psi'''|^2&\nonumber\\
&\qquad \left.+\frac{1}{2}K [n_1'^2+n_2'^2+n_3'^2]+\frac{PE}{\bs} |\psi|^2n_3\right\}\,dx&\nonumber
\end{flalign}
\remark We have added a regularizing factor $\frac{\rho}{q^6}|\psi'''|^2$ , with $\rho$ a small positive constant, to aid with the analysis. We will let $\rho\to 0$ later.
\section{Static Analysis}
To set up for the dynamic analysis, we highlight the quantities that are bounded uniformly (in $q$). We are able to find well-prepared initial data, specifically we can construct a family of possible initial configurations. One such example is $\n^0=(\co \sqrt{1+b^2}, 0, \sqrt{1-\cos^2\theta(1+b^2)})$ and $\psi^0=e^{-\frac{iq}{\bs}g(x_2)}$ where $g'(x_2)=-b\tanh (qx_2)/\tanh(qL).$
\begin{lemma}
For $q$ sufficiently large, there exists $(\n^0,\psi^0)\in X$ such that $\mathcal{F}(\n^0,\psi^0)\le C_0$, where $C_0$ is independent of $q$.
\label{InitialCond}
\end{lemma}

Our dynamic analysis is based on energy minimization, so we only consider the states $(\n,\psi)\in X$ with $\mathcal{F}(\n,\psi)\le \mathcal{F}(\n^0,\psi^0)$. Through out this paper it is assumed that the constants appearing in (\ref{Energy}), with the exceptions of $0\leq\rho<1$ and $q\geq1$ are fixed. We use $C_1$ to denote a constant in our estimates  such that $C_1(\mathcal{C})$ is independent of $\rho,$  states $(\n,\psi)$  for which  $\mathcal{F}(\n,\psi)\le \mathcal{C},$  and  all $q$  sufficiently large $q\geq q_0(\mathcal{C}).$ Since the energy for our initial data is uniformly bounded, we can deduce using Sobolev's embedding theorem in 1-dim that $|\psi|$ and $\frac{|\psi|^2{}''}{q^2}$ at later times are uniformly bounded  as well. We also prove a Modica-Mortola type estimate.
\begin{lemma}
For $q$ sufficiently large, $\dfrac{|\psi'|}{q}\le C_1$ on $[-L,L]$, where $C_1$ is independent of $q$.
\label{Bdpsi'}
\end{lemma}
This specific boundedness, in fact, has an important physical implication for our system: the coupling between $\n$ and $|\psi|$ weakens for a sufficiently large $q$.
\begin{proof} Note $\psi \in C^1[-L,L]$ and we seek a specific bound in terms of $q$. We write $\{x: |\psi(x)|>0, -L<x<L\}$ as a countable  union of disjoint intervals $\{(a_j,b_j) \text{ for }j\in\mathcal{I}\}.$ Since $\frac{|\psi'|}{q}=0$ on $ (-L,L)\text{\textbackslash}\overline{\underset{j\in\mathcal{I}}\cup (a_j,b_j)}$, it is enough to prove $\frac{|\psi'|}{q}\le C_1$ on each interval $(a_j,b_j)$ for a constant $C_1$ independent of $j\in\mathcal{I}.$ The plan is to prove the real and imaginary parts of $\frac{\psi'\cg}{q|\psi|}$ are bounded.


We first note that since  $\mathcal{F}(\n,\psi)\le \mathcal{F}(\n^0,\psi^0)$ we have
$$\int_{-L}^L((|\psi|^4+(|\psi|^2{}')^2+\frac{1}{q^2}(|\psi|^2{}'')^2+\frac{1}{q^6}(|\psi|^2{}''')^2)\,dx\leq C_1.$$
It follows that $|\frac{|\psi|^2{}''}{q^2}|+|\psi|^2\leq M_1$ where $C_1$ and $M_1$ are independent of $q$ for $q\geq 1$.
Fix $j\in \mathcal{I}.$ Using the fact that the initial energy is bounded, and after carrying out some algebraic manipulations, we get
\begin{flalign}
\int_{a_j}^{b_j}\frac{1}{q}\left|\Im\{\frac{\psi''\cg}{q}\}\right|^2|\psi|^{-2}+q\left|\Re\{\frac{\psi''\cg}{q^2}\}+ \frac{b^2}{1+b^2}|\psi|^2\right|^2|\psi|^{-2}\,dx\le C_1&
\label{MME}
\end{flalign}
\begin{flalign*}
q\int_{a_j}^{b_j}\left|\Re\{\frac{\psi''\cg}{q^2}\}+ \frac{b^2}{1+b^2}|\psi|^2\right|^2|\psi|^{-2}\,dx&=q\int_{a_j}^{b_j}\left|\left(\Re\{\frac{\psi'\cg}{q^2}\}\right)'-\frac{|\psi'|^2}{q^2}+ \frac{b^2}{1+b^2}|\psi|^2\right|^2|\psi|^{-2}\,dx&\\
&=q\int_{a_j}^{b_j}\left|\frac{|\psi|^2{}''}{2q^2}-\frac{|\psi'|^2}{q^2}+ \frac{b^2}{1+b^2}|\psi|^2\right|^2|\psi|^{-2}\,dx&
\end{flalign*}
Consider the set $E=\left\{x\in (a_j, b_j); \frac{|\psi'|^2}{q^2}\ge 2M\right\}$ where $M$ is such that $\frac{|\psi|^2{}''}{2q^2}+ \frac{b^2}{1+b^2}|\psi|^2\le M$. On $E$, we have
\begin{equation}
\frac{q}{4}\int_E\left[\frac{|\psi'\cg|^2}{q^2|\psi|^2}\right]^2|\psi|^{-2}\,dx=\frac{q}{4}\int_E\left[\frac{|\psi'|^2}{q^2}\right]^2|\psi|^{-2}\,dx\le C_1.
\label{IneqT}
\end{equation}
and
\begin{equation}
q\displaystyle\int_E\left[\frac{|\psi'|^2}{q^2}-2M\right]^2|\psi|^{-2}\,dx\le C_1.
\label{IneqM}
\end{equation}
We now utilize the equality
\begin{align*}
\left(\Re\{\frac{\psi'\cg}{q|\psi|}\}\right)'&=\left(\frac{-\Re\{\frac{\psi'\cg}{|\psi|}\}}{|\psi|^2}\right)\left(\Re\{\frac{\psi'\cg}{q}\}\right)+\frac{1}{|\psi|}\left(\Re\{\frac{\psi'\cg}{q}\}\right)'.&
\end{align*}
Knowing that
\begin{align*}
\frac{1}{q}\int_{a_j}^{b_j}\left[\left(\Re\{\frac{\psi'\cg}{q}\}\right)'\right]^2\,dx=\frac{1}{q}\int_{a_j}^{b_j}\left(\frac{|\psi|^2{}''}{2q}\right)^2\,dx\,dx\le C_1
\end{align*}
and that
\begin{align*}
\frac{1}{q}\int_E\left|\Re\{\frac{\psi'\cg}{|\psi|}\}\Re\{\frac{\psi'\cg}{q|\psi|}\}\right|^2\,dx\le qM_1\int_E\left[\frac{|\psi'\cg|^2}{q^2|\psi|^2}\right]^2|\psi|^{-2}\,dx\le C_1\text{ by (\ref{IneqT})},
\end{align*}
we conclude that
\begin{flalign}
\frac{1}{q}\int_E\left||\psi|\left(\Re\{\frac{\psi'\cg}{q|\psi|}\}\right)'\right|^2\le C_1.&
\label{IneqR}
\end{flalign}
Adding Inequalities (\ref{IneqM}) and (\ref{IneqR}),
\begin{flalign*}
C_1&\ge \frac{1}{q}\int_E\left||\psi|\left(\Re\{\frac{\psi'\cg}{q|\psi|}\}\right)'\right|^2\,dx+ q\displaystyle\int_E\left[\frac{|\psi'|^2}{q^2}-2M\right]^2|\psi|^{-2}\,dx&\\
&\ge \int_{a_j}^{b_j}\left|\left(\Re\{\frac{\psi'\cg}{q|\psi|}\}\right)'\left[\frac{|\psi'|^2}{q^2}-2M\right]^+\right|\,dx\ge \int_{a_j}^{b_j}\left|\left(\Re\{\frac{\psi'\cg}{q|\psi|}\}\right)'\left[\left(\Re\{\frac{\psi'\cg}{q|\psi|}\}\right)^2-2M\right]^+\right|\,dx,&
\end{flalign*}
noting that the second inequality above is a Modica-Mortola type estimate.\\
If $\Phi$ is such that $\Phi'(y)=[y^2-2M]^+$, then $\underset{(a_j,b_j)}{osc}\:\:\Phi\left(\frac{|\psi|'}{q}\right)=\underset{(a_j,b_j)}{osc}\:\:\Phi\left(\Re\{\frac{\psi'\cg}{q|\psi|}\}\right)\le C_1$. If $-L<a_j<b_j<L$ then $|\psi(a_j)|=|\psi(b_j)|=0$ and it follows that $|\psi(x)|'=0$ for some $x\in(a_j,b_j)$. If either $a_j=-L$ or $b_j=L$ then it follows from the boundary conditions that $|\psi|'=0$ at that point. In either case
 it follows that  $|\frac{|\psi|'}{q}|$ is uniformly bounded independent of $q$ and $j$ on $(a_j,b_j)$. A similar reasoning can be applied to the imaginary part to get the conclusion.
\end{proof}
One last static result we present is the existence of minimizers.
\begin{theorem}
For $q$ sufficiently large, there exists $(\mathbf{m}, \xi)\in X$ such that
$$\mathcal{F}(\mathbf{m}, \xi)=\inf_{(\n,\psi)\in X}\mathcal{F}(\n,\psi).$$
\label{MinExistence}
\end{theorem}
\begin{proof}
By coercivity and initial boundedness of the energy (\ref{Energy}), we can guarantee the existence of a subsequence $(\n_j,\psi_j)$ that converges weakly to $(\mathbf{m}, \xi)$ in $H^1\times H^3$. Weak convergence, however, is not enough due to the nonlinearity of the terms. We invoke a Sobolev embedding theorem to deduce that $\{\n_j\}$ and $\{\psi'_j\}$ are uniformly bounded in $C^{0,1/2}(-L,L)$, followed by Arzel\`{a}-Ascoli theorem to get uniform convergence to $\mathbf{m}$ and $\xi'$ respectively.
By convergence proved thus far and lower semicontinuity of the $L^2$-norm,
\begin{equation*}
\mathcal{F}(\mathbf{m}, \xi)\leq \liminf \mathcal{F}(\n_j, \psi_j),
\end{equation*}
proving that $\mathcal{F}(\mathbf{m}, \xi)$ is a minimizer.
\end{proof}
\section{Dynamic Analysis}
\subsection{Method of Rothe}
We begin the dynamic analysis of a chevron structure under an applied electric field by constructing a discretized-in-time gradient flow. We follow the Method of Rothe [23], through which we construct an approximate elliptic-type problem. Convergence of the approximate solution to the continuous solution is the main goal of our analysis. An advantage of this approach is that it not only exploits the variational feature of the problem, but also accommodates for its nonlinearity, as will be seen later.

 Let $(\n^0,\psi^0)\in X$ be any initial data (which we showed exist in lemma \ref{InitialCond}). Consider any time period $[0,T]$ and let $\tau>0$ be any step size in $t$. Choose the number of steps $M$ such that $M\tau>T$. Minimize:
\begin{align*}
J^0(\n,\psi)&=\int_{-L}^L \left\{\frac{|\n-\n^0|^2}{2\tau}+\frac{|\psi-\psi^0|^2}{2\tau}\right\}\,dx+\mathcal{F}(\n,\psi)
\end{align*}
where $\mathcal{F}(\n,\psi)=\displaystyle\int_{-L}^L\left\{f_N+\frac{1}{q}f_{CL}+f_E\right\}\,dx$, with the given initial values, on $[0,\tau]$. We know that such a minimizer exists by theorem \ref{MinExistence}. Denote the minimizer by $(\n^1,\psi^1)$. Use the minimizer $(\n^1,\psi^1)$ as the initial values for the second time step $[\tau,2\tau]$ and minimize the new energy functional $J^1(\n,\psi)$. Repeating the process, we get a sequence of minimizing problems and a family of minimizers $(\n^m,\psi^m)$, $m=0,1,...,M$.

These minimizers satisfy Euler-Lagrange equations as well as an energy dissipation inequality.
Specifically, the minimizer $(n_1^m,n_2^m,n_3^m,\psi^m)$ satisfies the following four equations over the time interval $((m-1)\tau,m\tau]$ -  see details of the derivation in Appendix A,

\begin{flalign}\label{sE-L1}
&\int_{-L}^L \left\{(1-n^2_1)\delta_\tau n_1-n_1n_2\delta_\tau n_2-n_1n_3\delta_\tau n_3+(1-n^2_1)F_{n_1}-n_1n_2F_{n_2}+(1-n^2_1)'F_{n'_1}\right.&\\
&\qquad\quad -(n_1n_2)'F_{n'_2}-K(n_1'^2+n_2'^2+n_3'^2)n_1\left.-\frac{PE}{\bs} |\psi|^2n_1n_3\right\}u_1\nonumber&\\
&\quad+\left\{(1-n^2_1)F_{n'_1}-n_1n_2F_{n'_2}+Kn_1'\right\}u'_1\,dx=0\nonumber&
\end{flalign}
\begin{flalign}\label{sE-L2}
&\int_{-L}^L\left\{ -n_1n_2\delta_\tau n_1+(1-n^2_2)\delta_\tau n_2-n_2n_3\delta_\tau n_3-n_1n_2F_{n_1}+(1-n^2_2)F_{n_2}-(n_1n_2)'F_{n'_1}\right.&\\
&\qquad\quad+(1-n^2_2)'F_{n'_2}-K(n_1'^2+n_2'^2+n_3'^2)n_2\left.-\frac{PE}{\bs} |\psi|^2n_2n_3\right\}u_2&\nonumber\\
&\quad+\left\{-n_1n_2F_{n'_1}+(1-n^2_2)F_{n'_2}+Kn_2'\right\} u'_2\,dx=0\nonumber&
\end{flalign}
\begin{flalign}\label{sE-L3}
&\int_{-L}^L\left\{ -n_1n_3\delta_\tau n_1-n_2n_3\delta_\tau n_2+(1-n^2_3)\delta_\tau n_3-n_1n_3F_{n_1}-n_2n_3F_{n_2}-(n_1n_3)'F_{n'_1}\right.&\\
&\qquad\quad-(n_2n_3)'F_{n'_2}-K(n_1'^2+n_2'^2+n_3'^2)n_3\left.+\frac{PE}{\bs} |\psi|^2(1-n_3^2)\right\}u_3&\nonumber\allowdisplaybreaks\\
&\qquad+\left\{-n_1n_3F_{n'_1}-n_2n_3F_{n'_2}+Kn_3'\right\}u'_3\,dx=0\nonumber&
\end{flalign}
and
\begin{align}\label{sE-L4}
2\Re \int_{-L}^L&\frac{1}{2}\delta_\tau \psi\overline{\phi}+F_\psi\overline{\phi}+F_{\psi'}\overline{\phi'}+F_{\psi''}\overline{\phi''}+2g(\psi\overline{\phi})(|\psi|^2-1)+2(\psi\overline{\phi})'|\psi|^2{}'&\\
&+\frac{2}{q^2}(\psi\overline{\phi})''|\psi|^2{}''+\frac{2}{q^6}(\psi\overline{\phi})'''|\psi|^2{}'''+\frac{\rho}{q^6}(\psi'''\overline{\phi'''})+\frac{PE}{\bs}n_3(\psi\overline{\phi})\,dx=0,&\nonumber
\end{align}
where $(u_1,u_2,u_3,\phi)$ are the test functions such that $ (u_1,u_2,u_3) \in H^1(-L,L),\;\phi \in H^3(-L,L)$ with $\phi \in H^2_0(-L,L) \text{ and }\Re\{\phi''\cg\}(-L)=\Re\{\phi''\cg\}(L)=0.$\\
The elastic part of the smectic free energy density from (\ref{Energy}) is
\begin{align*}
F(n_1,n_2,n_3,\psi)&=\frac{a_{\perp}}{q^3}|\psi''-[(\frac{iq}{\bs}n_1\psi+n_2\psi')n_2]'-\frac{q^2}{1+b^2}\psi +\frac{q^2}{1+b^2}n_1^2\psi&\\
&\qquad\quad\left.-\frac{iq}{\bs}n_1n_2\psi'+\frac{c_{\perp}q^2}{2a_{\perp}}\psi|^2\right.&\allowdisplaybreaks\\
&\quad+\frac{a_{\|}}{q^3}|(\frac{n_1}{\bs}-\co)(-\frac{q^2}{\bs}n_1 \psi +iqn_2\psi'+q^2\co \psi)&\\
&\qquad\quad+[(\frac{iq}{\bs}n_1\psi+n_2\psi'-iq\co \psi)n_2]'|^2&\\
&\quad+\frac{c_{\|}}{q}|\frac{iq}{\bs}n_1 \psi+n_2 \psi'-iq\co \psi|^2.&
\end{align*}
We have used the notation $F_{n_1}$ for $\partial_{n_1}F$ and the fact that $\partial_x|z|^2=2\Re\{\overline{z}_xz\}$ for a complex number $z$. However, in the fourth equation, for $F_\psi$ where $F=a_1|g_1|^2+a_2|g_2|^2+a_3|g_3|^2$ means $a_1\overline{g_{1\psi}}g_1+a_2\overline{g_{2\psi}}g_2+a_3\overline{g_{3\psi}}g_3$. In addition, $\delta _\tau n_1$ is the difference quotient defined by $\frac{1}{\tau}(n_1-n_1^{m-1}).$
\begin{lemma}{(Energy Dissipation)}
\begin{equation}
\frac{1}{2}\sum\limits_{k=1}^m \tau (||\delta _\tau \n^k||_{L^2(-L,L)}^2+||\delta _\tau \psi^k||_{L^2(-L,L)}^2)+\mathcal{F}(\n^m,\psi^m)\le \mathcal{F}(\n^0,\psi^0)\quad \text{for $1 \le m \le M$.}
\label{EnrgyDsp}
\end{equation}
\end{lemma}
\begin{proof}
Let $m$ be an integer such that $1\le m\le M$ where $M\tau > T$. Since $(\n^m,\psi^m)$ is a minimizer of $J^m(\n,\psi)$, we have
$$J^m(\n^m,\psi^m) \le J^m(\n^{m-1},\psi^{m-1}) = \mathcal{F}(\n^{m-1},\psi^{m-1}).$$
Adding the $(m-1)^{st}$ difference quotient to both sides of the above inequality,
\begin{align*}
&\int_{-L}^L \frac{|\n^{m-1}-\n^{m-2}|^2}{2\tau}+\frac{|\psi^{m-1}-\psi^{m-2}|^2}{2\tau}\,dx+J^m(\n^m,\psi^m) \le J^{m-1}(\n^{m-1},\psi^{m-1}).&
\end{align*}
With  $J^{m-1}(\n^{m-1},\psi^{m-1})\le \mathcal{F}(\n^{m-2},\psi^{m-2})$, we get

$\dfrac{1}{2}\sum\limits_{k=m-1}^m \tau (||\delta _\tau \n^k||_{L^2(-L,L)}^2+||\delta _\tau \psi^k||_{L^2(-L,L)}^2)+\mathcal{F}(\n^m,\psi^m)\le \mathcal{F}(\n^{m-2},\psi^{m-2}).$\\

\noindent Adding the $(m-2)^{nd}$ difference quotient to both sides of the above inequality and iterating, we deduce the desired inequality.
\end{proof}
Set $\Omega=(-L,L).$ To extend the Euler-Lagrange equations to $\Omega_T=(-L,L)\times(0,T)$, we construct piecewise constant (in $t$) functions, for instance, $n_1^\tau(x,t)=n_1^\tau(x,m\tau)=n_1^m(x)\text{ for } t\in ((m-1)\tau,m\tau], m=1,...,M$ for $n_1$; and similarly for the other components. For piecewise constant test functions, we multiply the first Euler-Lagrange equation (\ref{sE-L1}) by $\tau$ and add up the equations as $m$ spans $1$ to $M$. As the integrand is independent of $t$, we get:
\begin{flalign}
\label{ELtau}
&\int_0^T\int_{-L}^L \left\{(1-(n^\tau_1)^2)\delta _\tau n_1(x,t)-n^\tau_1n^\tau_2\delta _\tau n_2(x,t)-n^\tau_1n^\tau_3\delta _\tau n_3(x,t)+(1-(n^\tau_1)^2)F_{n^\tau_1}\right.&\\
& \qquad\qquad-n^\tau_1n^\tau_2F_{n^\tau_2}+(1-(n^\tau_1)^2)'F_{n_1^\tau{}'}-(n^\tau_1n^\tau_2)'F_{n_2^\tau{}'}-K[(n_1^\tau{}')^2+(n_2^\tau{}')^2+(n_3^\tau{}')^2]n^\tau_1&\nonumber\\
&\qquad\qquad\left.-\frac{PE}{\bs} |\psi^\tau|^2n^\tau_1n^\tau_3\right\}u_1(x,t)&\nonumber\\
& \qquad\quad+\left\{(1-(n^\tau_1)^2)F_{n_1^\tau{}'}-n^\tau_1n^\tau_2F_{n_2^\tau{}'}+Kn^\tau_1{}'\right\}u'_1(x,t)\,dxdt=0\nonumber&
\end{flalign}
We define the Sobolev-Bochner space $H^{k,1}(\Omega_T)$ by
$$H^{k,1}(\Omega_T)=\{u(x,t)\in L^2(\Omega_T)\,;\,\partial^j_xu(x,t)\in L^2(\Omega_T),0\le j\le k \text{ and }\partial_tu(x,t)\in L^2(\Omega_T)\}. $$
Since our piecewise constant functions fail to belong to such a Sobolev space, we construct piecewise linear (in $t$) functions $(\tilde{n}_1^\tau(x,t))$, for instance, $\tilde{n}_1^\tau(x,t)=\frac{t-(m-1)\tau}{\tau}n_1^m(x)+\frac{m\tau-t}{\tau}n_1^{m-1}(x), t\in [(m-1)\tau, m\tau) \text{  for } m=1,...,M$; and similarly for the other components.
\begin{lemma}
$\{\tilde{n}_1^\tau(x,t)\}$ is uniformly bounded in $H^{1,1}(\Omega_T)$ for any $\tau$. Similarly, $\{\tilde{n}_2^\tau(x,t)\}$ and $\{\tilde{n}_3^\tau(x,t)\}$ are uniformly bounded in $H^{1,1}(\Omega_T)$,\: $\{\widetilde{\psi}^\tau(x,t)\}$ is bounded in $H^{2,1}(\Omega_T)$, and $\{|\widetilde{\psi}^\tau(x,t)|^2\}$ is bounded in $H^{3,1}(\Omega_T)$ for any $\tau$.
\end{lemma}

\begin{proof} We have
\begin{flalign*}
&\partial_t\tilde{n}_1^\tau(x,t)=\frac{n_1^m(x)-n_1^{m-1}(x)}{\tau},\quad t\in ((m-1)\tau, m\tau) \text{  for } m=1,...,M.&
\end{flalign*}
\begin{flalign*}
\int_0^T\int_{-L}^L|\partial_t\tilde{n}_1^\tau(x,t)|^2\,dxdt&\le \sum_{m=1}^M\int_{(m-1)\tau}^{m\tau}\int_{-L}^L\left|\frac{n_1^m(x)-n_1^{m-1}(x)}{\tau}\right|^2\,dxdt&\\
&=\sum_{m=1}^M\tau\int_{-L}^L\left|\frac{n_1^m(x)-n_1^{m-1}(x)}{\tau}\right|^2\,dx\le C_1,&
\end{flalign*}
where that last inequality is true by (\ref{EnrgyDsp}). From the well-prepared initial data, we get the uniform bound.
In a similar way, we can prove $\int_0^T\int_{-L}^L|\nabla_x\tilde{n}_1^\tau(x,t)|^2\,dxdt$ is uniformly bounded.
\end{proof}

Since $\{\tilde{n}_1^\tau\}$ is uniformly bounded in $H^{1,1}(\Omega_T)$, there is a subsequence, still denoted by $\{\tilde{n}_1^\tau\},$ which converges strongly to some $n_1$ in $L^2(\Omega_T)$. Also, by weak compactness of Sobolev spaces we can find a subsequence, $\{\tilde{n}_1^\tau\}$ converging weakly to $n_1$ in $H^{1,1}(\Omega_T).$ Similarly, we can prove strong convergence of $\tilde{n}_2^\tau$, $\tilde{n}_3^\tau$, $\widetilde{\psi}^\tau$, and $|\widetilde{\psi}^\tau|^2$; and weak convergence in the corresponding Sobolev-Bochner spaces.

In the analysis of the discrete gradient flow, it is more convenient to work with the piecewise constant approximations than with the  piecewise linear approximations and since the two have the same asymptotic behavior, we use $(n_1^\tau,n_2^\tau,n_3^\tau,\psi^\tau)$ from now on.

\subsection{Convergence of the Discrete Gradient Flow}

Due to the high nonlinearity of the discrete system, the above convergence of subsequences is not enough to prove convergence of the discrete gradient flow. It turns out that we need higher regularity, which is achieved through the following three major steps that are carried out for $q$  sufficiently large.

{\bf Step I. Local Regularity.} We utilize the Euler-Lagrange equations themselves, and replace the test functions by convenient ones. For instance, in equation (\ref{ELtau}), we let $u_1=\Delta_{-h}[(\Delta_hn^\tau_1)\varphi^2]$ for a small $h>0$, where $\varphi\in C_0^\infty(-L,L)$ is a cut-off function s.t. $0\le\varphi\le 1$ and
\begin{equation}\varphi(x) = \left\{
     \begin{array}{lr}
       1 & \text{if } x \in (-L+\eta,L-\eta)\\
       0 & \text{if } x \notin (-L+\frac{\eta}{2},L-\frac{\eta}{2})
     \end{array}
   \right.\text{ for } \eta>0.\label{cut-off}\end{equation} From the resulting equation, we are able to deduce an estimate on $\int_{\Omega_T}|\Delta_hn_1^\tau{}'\varphi|^2\,dxdt$ in terms of the initial energy bound and small multiples of the integral itself (See Appendix B). We repeat the process for the remaining Euler-Lagrange equations, as the equations are coupled, and get the following estimate
\begin{flalign*}
\int_0^T\int_{-L+\eta}^{L-\eta}\left(|\Delta_h\n^\tau{}'|^2+|\Delta_h\psi^\tau{}''|^2+|\Delta_h|\psi^\tau|^2{}'''|^2+\rho\right.&\left.|\Delta_h\psi^\tau{}'''|^2\right)\varphi^2\,dxdt\le C(\rho,q).
\end{flalign*}
Letting $h\to 0$,
\begin{equation*}
\int_0^T\int_{-L+\eta}^{L-\eta}|\n^\tau{}''|^2+|\psi^\tau{}'''|^2+||\psi^\tau|^2{}^{(4)}|^2+\rho|\psi^{\tau(4)}|^2\,dxdt\le C(\rho,q).
\end{equation*}
We are able to remove the dependence on $\rho$ in the final estimate (See Appendix B) to get the following theorem.
\begin{theorem}
\begin{equation}
\int_{\Omega'_T}|\n^\tau{}''|^2+|\psi^\tau{}'''|^2+||\psi^\tau|^2{}^{(4)}|^2+\rho|\psi^{\tau(4)}|^2\,dxdt\le C(q).
\label{LocEst}
\end{equation}
where $\Omega'_T=(L-\eta,L+\eta)\times [0,T].$
\end{theorem}

{\bf Step II. Higher Local Regularity.} Due to the nature of the local regularity achieved thus far, we are not able to attain the convergence required. We therefore prove higher local regularity by following a similar method. Specifically, the test function in the first equation is now replaced by $u_1=\Delta_{-h}[(\Delta_hn^\tau{}'_1)\varphi^2]'$ for a small $h>0$ and with the same cut-off function as before. The difference here is that we require more of the initial conditions on $\n$ (See Appendix C). We complete the estimates by removing the dependence on $\rho$, as before, to get the following result.
\begin{theorem}
Under the assumption that $\displaystyle\int_{\Omega'}|\n^0{}''|^2\,dx$ is initially bounded, we have
\begin{equation}
\int_{\Omega'_T}|\n^\tau{}'''|^2+|\psi^{\tau(4)}|^2+||\psi^\tau|^2{}^{(5)}|^2+\rho|\psi^{\tau(5)}|^2\,dxdt\le C(q).
\label{HigherLocEst}
\end{equation}
\end{theorem}
To be able to extend the estimates to the full domain, we need to get rid of the regularization term at this stage. The regularity results, (\ref{LocEst}) and (\ref{HigherLocEst}), are set up in a way that allows for this, due to their independence of $\rho$. However, to recover the Euler-Lagrange equations when we let $\rho\to 0$, more should be done (See Appendix D). Specifically, the fourth Euler-Lagrange equation recovered is
\begin{align}
\label{E-L4}
2\Re \int_0^T\int_{-L}^L&\frac{1}{2}\delta_\tau \psi\overline{\phi}+F_{\psi^\tau}\overline{\phi}+F_{\psi^\tau{}'}\overline{\phi'}+F_{\psi^\tau{}''}\overline{\phi''}+2g(\psi^\tau\overline{\phi})(|\psi^\tau|^2-1)+2(\psi^\tau\overline{\phi})'|\psi^\tau|^2{}'&\\
&+\frac{2}{q^2}(\psi^\tau\overline{\phi})''|\psi^\tau|^2{}''+\frac{2}{q^6}(\psi^\tau\overline{\phi})'''|\psi^\tau|^2{}'''+\frac{PE}{\bs}n_3(\psi^\tau\overline{\phi})\,dxdt=0.\nonumber\\\nonumber&
\end{align}

{\bf Step III. Regularity up to the Boundary.}

\begin{theorem}
\begin{equation}
\int_{\Omega_T}|\n^\tau{}''|^2+|\psi^\tau{}'''|^2+||\psi^\tau|^2{}^{(5)}|^2\,dxdt\le C(q).\\
\label{BdryEst}
\end{equation}
\end{theorem}
\vspace{0.1in}

\begin{proof}
Our analysis will deal with the right boundary point $x=L$, the left boundary point is analogous. Since $|\psi^\tau|^2(L,t)=1$ and $\int_{-L}^L(|\psi^\tau(x,t)|^2{}')^2 dx\leq C_1$ uniformly in $t$, we can assume that $|\psi^\tau|^2>1/2$ over $(L-\lambda, L)$ for some $\lambda>0.$ Fix $L-\frac{\lambda}{2}<L'<L$. In the following, we let $R_j$ denote a sum of terms that are integrable over $(L-\lambda,L')$ and whose square integrals are bounded by a constant and small multiples of $\int_{L-\lambda}^{L'}|n_1^\tau{}''|^2\,dx$, $\int_{L-\lambda}^{L'}|n_2^\tau{}''|^2\,dx$, $\int_{L-\lambda}^{L'}|n_3^\tau{}''|^2\,dx$ and $\int_{L-\lambda}^{L'}|\psi^\tau{}'''|^2\,dx$. And let $S_j$ denote a sum of terms such as $|\delta_\tau n_1|$, $|\delta_\tau n_2|$, $|\delta_\tau n_3|$, $|\delta_\tau \psi|$ and their integrals over $(L-\lambda, L).$

We replace the test function in (\ref{E-L4}) by $\phi=\zeta\frac{\psi^\tau}{|\psi^\tau|^2}$ where $\zeta$ is a compactly supported smooth function over $[0,T]\times (L-\lambda,L).$ The higher estimates (\ref{HigherLocEst}), together with integration by parts, allow us to write the Euler-Lagrange equation $(\ref{E-L4})$ in explicit form,
\begin{align}
\label{NormOrd5}
|\psi^\tau|^{2(5)}&=2\Re\left\{\frac{a_{\perp}}{q^2}(1-n_2^{\tau2})\left((1-n_2^{\tau2})\frac{\psi^\tau{}'}{q}\right)''\frac{\overline{\psi^\tau}}{|\psi^\tau|^2}+\frac{a_{\|}}{q^2}n_2^{\tau2}\left(n_2^{\tau2}\frac{\psi^\tau{}'}{q}\right)''\frac{\overline{\psi^\tau}}{|\psi^\tau|^2}\right\}&\\
&\quad+R_1+S_1 \text{ for } L-\delta<y<L.&\nonumber
\end{align}
Now we go back to the first three weak Euler-Lagrange equations, write them in explicit form, and deduce the following estimate:
\begin{equation}
|n_1^\tau{}''|+|n_2^\tau{}''|+|n_3^\tau{}''|\le C|\psi^\tau{}'''|+|R_2|+|S_2|
\label{nEst}
\end{equation}
We also go back to the weak equation (\ref{E-L4}), integrate by parts, and write the explicit equation. We replace $\psi^\tau|\psi^\tau|^{2(6)}$ by $(\psi^\tau|\psi^\tau|^{2(5)})'-\psi^\tau{}'|\psi^\tau|^{2(5)}$, plug equation (\ref{NormOrd5}) into (\ref{E-L4}), and then take anti-derivatives to get the estimate
\begin{equation}
|\psi^\tau{}'''|\le\epsilon|n_1^\tau{}''|+\epsilon|n_2^\tau{}''|+\int_{L-\lambda}^x\epsilon|n_1^\tau{}''|+\int_{L-\lambda}^x\epsilon|n_2^\tau{}''|+C\int_{L-\lambda}^x|\psi^\tau{}'''|+|R_3|+|S_3|.
\end{equation}
Provided $q$ is sufficiently large we have $\epsilon$ small and  can insert inequality (\ref{nEst}) into  the above to get
\begin{equation}
\label{PsiInt}
|\psi^\tau{}'''(x)|\le C\int_{L-\lambda}^x|\psi^\tau{}'''|+|R_4|+|S_4|.
\end{equation}
So $\left[e^{-Cx}\int_{L-\lambda}^x|\psi^\tau{}'''|\right]'\le e^{-Cx}|R_4|+e^{-Cx}|S_4|$. Integrating from $L-\lambda$ to $L'$, we get an estimate on $\int_{L-\lambda}^{L'}|\psi^\tau{}'''|$, which we use to bound the integral on the right-hand side of $(\ref{PsiInt}).$ We square the resulting inequality, as well as $(\ref{nEst})$, and integrate both from $L-\lambda$ to $L'$. The result is the following inequality
\begin{equation}
\int_{L-\lambda}^{L'}|n_1^\tau{}''|^2+|n_2^\tau{}''|^2+|n_3^\tau{}''|^2+ |\psi^\tau{}'''|^2\,dx\le C\int_{L-\lambda}^{L'}|\delta_\tau n_1|^2+|\delta_\tau n_2|^2+|\delta_\tau n_3|^2+|\delta_\tau \psi|^2\,dx+\overline C
\end{equation}
where crucially $\overline C$ is independent  $L'$. Integrating from $0$ to $T$, and since the resulting right-hand side is bounded by (\ref{EnrgyDsp}), we can let $L'\uparrow L$ to conclude that
\begin{equation}
\int_0^T\int_{L-\lambda}^L|n_1^\tau{}''|^2+|n_2^\tau{}''|^2+|n_3^\tau{}''|^2+ |\psi^\tau{}'''|^2\,dxdt\le C.
\end{equation}
Going back to (\ref{NormOrd5}), it is now easy to see that $\int_0^T\int_{L-\lambda}^L||\psi^\tau|^2{}^{(5)}|^2\,dxdt\le C$. Note that we gain one more order of regularity for $|\psi^\tau|^2$ through this method.
\end{proof}

\subsection{Existence} With the higher regularity bound (\ref{BdryEst}) thus obtained, we define the solution set $X_{sol}$ then state and prove the existence theorem. Let
\begin{align*}
X_{sol}=&\left\{(\n(x,t),\psi(x,t)) \in X;\right.&\allowdisplaybreaks\\
& \underset{t\in[0,T]}{\text{ ess sup }}\int_{-L}^L |\n|^2+|\n_x|^2+|\psi|^2+...+|\psi_{xx}|^2+||\psi|^2|^2+...+||\psi|^2_{xxx}|^2\,dx<\infty,&\\
&\int_{\Omega_T} |\n|^2+|\n_x|^2+|\n_{xx}|^2+|\psi|^2+...+|\psi_{xxx}|^2+||\psi|^2|^2+...+||\psi|^2_{xxxxx}|^2\,dxdt<\infty,&\\
&\left.\text{ and }\int_{\Omega_T} |\n_t|^2+|\psi_t|^2\,dxdt<\infty\right\}&
\end{align*}

\begin{theorem}
Given $\mathcal{C}>0$. We can find $q_0(\mathcal{C})$ so that, if $q>q_0$ and $\mathcal{F}_q(\n^0,\psi^0)\le \mathcal{C}$ for some initial data $(\n^0,\psi^0)$ with $\n^0{}''\in L^2(\Omega')$, there exists a solution $(\n(x,t), \psi(x,t))\in X_{sol}$ to the time-dependent Euler-Lagrange equations:
\begin{flalign}
\label{ContEL-n1}
&\int_{\Omega_T} \left\{\partial_tn_1+(1-n^2_1)F_{n_1}-n_1n_2F_{n_2}+(1-n^2_1)'F_{n'_1}-(n_1n_2)'F_{n'_2}-K(n_1'^2+n_2'^2+n_3'^2)n_1\right.&\\
& \qquad\left.-\frac{PE}{\bs} |\psi|^2n_1n_3\right\}u_1+\left\{(1-n^2_1)F_{n'_1}-n_1n_2F_{n'_2}+Kn_1'\right\}u'_1\,dxdt=0&\nonumber
\end{flalign}
\begin{flalign}
\label{ContEL-n2}
&\int_{\Omega_T} \left\{\partial_tn_2-n_1n_2F_{n_1}+(1-n^2_2)F_{n_2}-(n_1n_2)'F_{n'_1}+(1-n^2_2)'F_{n'_2}-K(n_1'^2+n_2'^2+n_3'^2)n_2\right.&\\
&\qquad\left.-\frac{PE}{\bs} |\psi|^2n_2n_3\right\}u_2+\left\{-n_1n_2F_{n'_1}+(1-n^2_2)F_{n'_2}+Kn_2'\right\}u'_2\,dxdt=0&\nonumber
\end{flalign}
\begin{flalign}
\label{ContEL-n3}
&\int_{\Omega_T} \left\{\partial_tn_3-n_1n_3F_{n_1}-n_2n_3F_{n_2}-(n_1n_3)'F_{n'_1}-(n_2n_3)'F_{n'_2}-K(n_1'^2+n_2'^2+n_3'^2)n_3\right.&\\
&\qquad\left.+\frac{PE}{\bs} |\psi|^2(1-n_3^2)\right\}u_3+\left\{-n_1n_3F_{n'_1}-n_2n_3F_{n'_2}+Kn_3'\right\}u'_3\,dxdt=0&\nonumber
\end{flalign}
\begin{flalign}
\label{ContEL-psi}
2\Re\int_{\Omega_T}&\frac{1}{2}\partial_t\psi\overline{\phi}+F_\psi\overline{\phi}+F_{\psi'}\overline{\phi'}+F_{\psi''}\overline{\phi''}+2g(\psi\overline{\phi})(|\psi|^2-1)+2(\psi\overline{\phi})'|\psi|^2{}'+\frac{2}{q^2}(\psi\overline{\phi})''|\psi|^2{}''&\allowdisplaybreaks\\
&\qquad+\frac{2}{q^6}(\psi\overline{\phi})'''|\psi|^2{}'''+\frac{PE}{\bs}n_3\psi\overline{\phi}\,dxdt=0&\nonumber
\end{flalign}
for any $(\mathbf{u},\phi)\in H^{1,1}(\Omega_T;\mathbb{R}^3)\times H^{3,1}(\Omega_T;\mathbb{C})$ such that  $\phi (\cdot,t)\in H^2_0(-L,L)$ with
$\Re\{\phi''\cg\}(-L,t)=\Re\{\phi''\cg\}(L,t)=0$ for almost every $0<t<T$.
\end{theorem}

\begin{proof}
Let $(\mathbf{n}^\tau,\psi^\tau)$ be a discrete gradient flow.
Write equation (\ref{ELtau}) as
\begin{flalign*}
\int_{\Omega_T} \{(1-n^\tau_1{}^2)\delta_\tau n_1-n^\tau_1n^\tau_2\delta_\tau n_2-n^\tau_1n^\tau_3\delta_\tau n_3+U^\tau_1\}u_1+V^\tau_1u'_1\,dxdt=0,&
\end{flalign*}
where $U^\tau_1$ and $V^\tau_1$ are nonlinear functions of $n^\tau_1, n^\tau_2, n^\tau_3, \psi^\tau$ and their derivatives. Knowing that $n^\tau_1\to n_1$ and $\delta_\tau n_1\rightharpoonup\partial_t n_1$  in $L^2(\Omega_T)$, with $\n\n_t=0$ since $|\n|=1$,
\begin{flalign*}
&\lim_{\tau\to 0}\int_{\Omega_T} \{(1-n^\tau_1{}^2)\delta_\tau n_1-n^\tau_1n^\tau_2\delta_\tau n_2-n^\tau_1n^\tau_3\delta_\tau n_3\}u_1\,dxdt&\\
&=\int_{\Omega_T} \{\partial_t n_1-n_1(n_1\partial_t n_1+n_2\partial_t  n_2+n_3\partial_t  n_3)\}u_1\,dxdt=\int_{\Omega_T} \partial_t n_1u_1\,dxdt.
\end{flalign*}
 It remains to prove that $U^\tau_1\rightharpoonup U_1$ and $V^\tau_1\rightharpoonup V_1$ in $L^2(\Omega_T)$. In fact, it suffices to prove that $U^\tau_1$, $V^\tau_1$ are bounded in $L^2(\Omega_T)$ and $U^\tau_1\to U_1$, $V^\tau_1\to V_1$ in $L^1(\Omega_T)$.
  To show how this can be achieved, we consider a typical nonlinear term, $Cn^\tau_2{}'\psi^\tau\psi^\tau{}''$. Before estimating this, note that it follows from (\ref{EnrgyDsp}) and the coercivity of $\mathcal{F}$ that
  $$\underset{t\in[0,T]}{\text{ ess sup }}\int_\Omega |\n^\tau{}'|^2+|\psi^\tau|^2+|\psi^\tau{}''|^2\,dx\leq C$$
  uniformly in $\tau$. Applying Nirenberg's interpolation inequality [24], we can see that
\begin{flalign*}
\int_{\Omega_T} |n^\tau_2{}'\psi^\tau\psi^\tau{}''|^2&\,dxdt\le C\int_0^T||\psi^\tau{}''||_{L^\infty(\Omega)}^2\int_{\Omega} |n^\tau_2{}'|^2\,dxdt\le C\int_0^T ||\psi^\tau{}''||_{L^\infty(\Omega)}^2dt&\\
&\le C\int_0^T \{C||\psi^\tau{}''||_{L^2(\Omega)}||\psi^\tau{}'''||_{L^2(\Omega)}+C||\psi^\tau{}''||_{L^2(\Omega)}^2\}dt&\\
&\le C\left[\int_0^T ||\psi^\tau{}''||^2_{L^2(\Omega)}\,dt\right]^{1/2}\left[\int_0^T||\psi^\tau{}'''||^2_{L^2(\Omega)}\,dt\right]^{1/2}+C||\psi^\tau{}''||_{L^2(\Omega_T)}^2&\allowdisplaybreaks\\
&\le C||\psi^\tau{}''||_{L^2(\Omega_T)}||\psi^\tau{}'''||_{L^2(\Omega_T)}+C||\psi^\tau{}''||_{L^2(\Omega_T)}^2< \infty,&
\end{flalign*}
where the last inequality is true by  (\ref{BdryEst}). This proves $L^2$-boundedness. To prove $L^1$-convergence, we utilize Nirenberg's interpolation inequality and higher regularity again to improve the convergence of subsequences.$$||n_1^{\tau_k}{}'-n_1^{\tau_j}{}'||_{L^2(\Omega_T)}^2\le C||n_1^{\tau_k}{}''-n_1^{\tau_j}{}''||_{L^2(\Omega_T)}||n_1^{\tau_k}-n_1^{\tau_j}||_{L^2(\Omega_T)}+C||n_1^{\tau_k}-n_1^{\tau_j}||_{L^2(\Omega_T)}^2.$$
$\{n_1^\tau\}$ is Cauchy and $\{n_1^\tau{}''\}$ is bounded in $L^2(\Omega_T)$ so $\{n_1^\tau{}'\}$ is Cauchy in $L^2(\Omega_T)$, hence convergent to $n_1'$. Similarly, we obtain $\psi^\tau{}''\to\psi'''$ in $L^2(\Omega_T).$ We have
\begin{flalign*}
\int_{\Omega_T} |n^\tau_2{}'\psi^\tau\psi^\tau{}''-n_2'\cg&\psi''|\,dxdt\le C \left[||\psi^\tau{}''||_{L^2(\Omega_T)}||n^\tau_2{}'-n_2'||_{L^2(\Omega_T)}\right.\\
&\left.+||n_2'\psi^\tau{}''||_{L^2(\Omega_T)}||\psi^\tau-\psi||_{L^2(\Omega_T)}+||n'_2||_{L^2(\Omega_T)}||\psi^\tau{}''-\psi''||_{L^2(\Omega_T)}\right]&
\end{flalign*}
and
\begin{flalign*}
\int_0^T\int_{\Omega}|n_2'\psi^\tau{}''|^2\,dxdt&=\sum_{m=1}^M\int_{(m-1)\tau}^{m\tau}\int_{\Omega'}|n_2'\psi^m{}''|^2\,dxdt\leq \sum_{m=1}^M\tau\sup_{\Omega}|\psi^m{}''|^2\int_{\Omega}|n_2'|^2\,dx&\\
&\leq C\sum_{m=1}^M\tau \sup_{\Omega}|\psi^m{}''|^2\leq C\sum_{m=1}^M\tau\left(\int_{\Omega}|\psi^m{}'''|^2\,dx+\int_{\Omega}|\psi^m{}''|^2\,dx\right)&\\
&= C\int_0^T\int_{\Omega}|\psi^m{}'''|^2\,dxdt+C\int_0^T\int_{\Omega}|\psi^m{}''|^2\,dxdt\leq C.
\end{flalign*}
So
\begin{flalign*}
&\int_{\Omega_T} |n^\tau_2{}'\psi^\tau\psi^\tau{}''-n_2'\cg\psi''|\,dxdt&\\
&\le C \left[||n^\tau_2{}'-n_2'||_{L^2(\Omega_T)}+||\psi^\tau-\psi||_{L^2(\Omega_T)}+||\psi^\tau{}''-\psi''||_{L^2(\Omega_T)}\right]\to 0\quad \text{as $\tau\to 0.$}&
\end{flalign*}
\end{proof}
The higher convergence obtained in the proof of theorem (4.6), along with energy dissipation statement (\ref{EnrgyDsp}), result in the following energy inequality.
\begin{corollary}
 \begin{equation}
\frac{1}{2}\int_{\Omega_s}|\partial_t\n|^2+|\partial_t\psi|^2\,dxdt+\mathcal{F}(\n(s),\psi(s) )\le \mathcal{F}(\n^0,\psi^0)\quad\text{ for }\; 0\leq s\le T.
\label{EnergyIneq}
\end{equation}
\end{corollary}
\subsection{Uniqueness}
 Now that we established the existence of a continuous gradient flow, we prove uniqueness of the solution independent of the choice of minimizing sequence and the time discretization.
\begin{theorem}
Given $\mathcal{C}>0$. We can find $q_1(\mathcal{C})$ so that, if $q>q_1$ and $\mathcal{F}_q(\n^0,\psi^0)\le \mathcal{C}$ for some initial data $(\n^0,\psi^0)$  then there exists at most one solution $(\n(x,t),\psi(x,t))\in X_{sol}$ to the time-dependent Euler-Lagrange equations, (\ref{ContEL-n1}), (\ref{ContEL-n2}), (\ref{ContEL-n3}), and (\ref{ContEL-psi}), satisfying the energy inequality (\ref{EnergyIneq}).
\label{Uniqueness}
\end{theorem}
\begin{proof}
We consider two solutions $(\n,\psi)$ and $(\tilde{\n},\widetilde{\psi})$ of the weak Euler-Lagrange equations with the same initial data $(\n^0,\psi^0)$ and that satisfy the energy inequality. We take the difference between the corresponding Euler-Lagrange equations and highlight the terms we need.
\begin{align*}
\int_{\Omega_T} &\left\{(\partial_t n_1-\partial_t \tilde{n}_1)+A\right\}u_1(x,t)+\left\{B+Kn_1'-K\tilde{n}_1'\right\}u'_1(x,t)\,dxdt=0.&
\end{align*}
Replace $u_1(x,t)$ by $v_1(x,t)\chi_{(t-\delta,t+\delta)}$ for each $t\in [0,T]$ then let $\delta\to 0$. By the Lebesgue Differentiation Theorem, we get for a.e. $t\in [0,T]$:
\begin{align*}
\int_{\Omega} &\left\{(\partial_t n_1-\partial_t \tilde{n}_1)+A\right\}v_1(x,t)+\left\{B+Kn_1'-K\tilde{n}_1'\right\}v'_1(x,t)\,dx=0.&
\end{align*}
Letting $v_1(x,t)=n_1-\tilde{n}_1$,
\begin{align*}
&\int_{\Omega} \frac{1}{2}\partial_t (n_1-\tilde{n}_1)^2+K(n_1'-\tilde{n}_1')^2=\int_{\Omega}-A(n_1-\tilde{n}_1)-B(n'_1-\tilde{n}'_1)\,dx.&
\end{align*}
We want to estimate the right-hand side, so we consider one of the terms,
 \begin{flalign*}
 &\frac{1}{q^2}\int_{\Omega}|(n'_2\cg\psi''-\tilde{n}'_2\overline{\widetilde{\psi}}\widetilde{\psi}'')(n_1-\tilde{n}_1)|\,dx&\\
 & =\frac{1}{q^2}\int_{\Omega}|n'_2\cg(\psi''-\widetilde{\psi}'')(n_1-\tilde{n}_1)+\widetilde{\psi}''\overline{\widetilde{\psi}}(n'_2-\tilde{n}'_2)(n_1-\tilde{n}_1)-\widetilde{\psi}''n'_2(\cg-\overline{\widetilde{\psi}})(n_1-\tilde{n}_1)|\,dx&\\
 & \le\frac{||n'_2\psi||_\infty}{q^2}\int_{\Omega}|\psi''-\widetilde{\psi}''||n_1-\tilde{n}_1|\,dx+\frac{||\widetilde{\psi}''\widetilde{\psi}||_\infty}{q^2}\int_{\Omega}|n'_2-\tilde{n}'_2||n_1-\tilde{n}_1|\,dx&\\
 &\qquad+\frac{||\widetilde{\psi}''n'_2||_\infty}{q^2}\int_{\Omega}|\psi-\widetilde{\psi}||n_1-\tilde{n}_1|\,dx&\allowdisplaybreaks\\
  & \le\frac{||\psi||^2_\infty}{2q^4}\int_{\Omega}|\psi''-\widetilde{\psi}''|^2\,dx+\frac{||n'_2||^2_\infty}{2q^4}\int_{\Omega}|n_1-\tilde{n}_1|^2\,dx+\frac{||\widetilde{\psi}||^2_\infty}{2q^4}\int_{\Omega}|n'_2-\tilde{n}'_2|^2\,dx&\\
  &+\frac{||\widetilde{\psi}''||^2_\infty}{2q^4}\int_{\Omega}|n_1-\tilde{n}_1|^2\,dx+\frac{||\widetilde{\psi}''||^2_\infty}{2q^4}\int_{\Omega}|\psi-\widetilde{\psi}|^2\,dx+\frac{||n'_2||^2_\infty}{2q^4}\int_{\Omega}|n_1-\tilde{n}_1|^2\,dx.&
 \end{flalign*}
Note that $|\psi|$ and $|\widetilde{\psi}|$ are uniformly bounded independent of $q$, in the first and third terms above. We repeat the process for the remaining equations and add up the final estimates to get, for $q$ sufficiently large,
\begin{align*}
&\frac{1}{2}\partial_t\int_{\Omega}|n_1-\tilde{n}_1|^2+|n_2-\tilde{n}_2|^2+|n_3-\tilde{n}_3|^2+|\psi-\widetilde{\psi}|^2\,dx&\\
&+\epsilon\int_{\Omega}|n_1'-\tilde{n}_1'|^2+|n_2'-\tilde{n}_2'|^2+|n_3'-\tilde{n}_3'|^2+|\psi''-\widetilde{\psi}''|^2\,dx&\\
&\le C\eta(t)\int_{\Omega}|n_1-\tilde{n}_1|^2+|n_2-\tilde{n}_2|^2+|n_3-\tilde{n}_3|^2+|\psi-\widetilde{\psi}|^2\,dx
\end{align*}
where
\begin{flalign*}
\eta(t)&=||n'_1||^2_\infty+||\tilde{n}'_1||^2_\infty+||n'_2||^2_\infty+||\tilde{n}'_2||^2_\infty+||n'_3||^2_\infty+||\tilde{n}'_3||^2_\infty+||\psi''||^2_\infty+||\widetilde{\psi}''||^2_\infty&\\
&+|||\psi|^2{}''||^2_\infty+|||\widetilde{\psi}|^2{}''||^2_\infty+|||\psi|^2{}'||^2_\infty+|||\widetilde{\psi}|^2{}'||^2_\infty+|||\psi|^2{}^{(4)}||^2_\infty+|||\widetilde{\psi}|^2{}^{(4)}||^2_\infty+C_1.
\end{flalign*}
As a consequence of the global estimates (\ref{BdryEst}), $\eta(t)$ is integrable over $[0,T]$. By the differential form of Gr\"{o}nwall's inequality, we conclude that
$$\int_{\Omega}|n_1-\tilde{n}_1|^2+|n_2-\tilde{n}_2|^2+|n_3-\tilde{n}_3|^2+|\psi-\widetilde{\psi}|^2\,dx=0\quad \text{for any $t\in[0,T]$}.$$
Consequently, $n_1-\tilde{n}_1=n_2-\tilde{n}_2=n_3-\tilde{n}_3=\psi-\widetilde{\psi}=0.$
\end{proof}
\section{Conclusions}
Our analysis of the switching dynamics of chiral smectic-C falls within the Landau-de Gennes theory. It also serves as an example of investigating defects using geometric flows, particularly in problems arising in Materials Science. We build our mathematical model upon the free energy introduced by Chen and Lubensky [22]. This energy is closely related to the work of Vaupoti\v{c}, Kralj, \v{C}opi\v{c} and Sluckin in [17], and Shalaginov, Hazelwood and Sluckin in [25]. The main difference is that $|\psi|=1$ is assumed throughout their analyses. While this does not present a major obstacle in the statics, it results in a high energy barrier to overcome in the dynamics. The construction and an analysis of the gradient flow under the assumption $|\psi|=1$ was done by Cheng in [26]. In the present work, we use a full complex-valued order parameter allowing the smectic structure to relax, which leads to a well-posed flow problem for a realistic energy barrier.\\

\noindent {\bf Appendices}\\
\vspace{-0.3in}

\begin{appendix}
\renewcommand{\thesubsection}{\Alph{subsection}}
\renewcommand{\theequation}{\Alph{subsection}.\arabic{equation}}
\subsection{Derivation of the Weak Euler-Lagrange Equations}
\noindent Write the Energy Functional as
\begin{align*}
J(\n,\psi) ={}&J_1(\n,\psi)+J_2(\n)+J_3(\n,\psi)+J_4(\n,\psi)+J_5(\psi)+J_6(\psi''')&\\
={}&\int_{-L}^L E(\n,\psi)\,dx+\int_{-L}^L H(\n)\,dx+\int_{-L}^L I(\n,\psi)\,dx+\int_{-L}^L F(\n,\psi)\,dx&\\
&+\int_{-L}^L G(|\psi|^2)\,dx+\int_{-L}^L \gamma|\psi'''|^2\,dx&
\end{align*}
where
$E=\displaystyle\frac{(n_1-n_1^0)^2}{2\tau}+\frac{(n_2-n_2^0)^2}{2\tau}+\frac{(n_3-n_3^0)^2}{2\tau}+\frac{|\psi-\psi^0|^2}{2\tau},$

\noindent $H=\dfrac{1}{2}K[n_1'^2+n_2'^2+n_3'^2], \quad I=\dfrac{PE}{\bs} |\psi|^2n_3,$
\begin{flalign*}
F&=\frac{a_{\perp}}{q^3}|\psi''-[(\frac{iq}{\bs}n_1\psi+n_2\psi')n_2]'-\frac{q^2}{1+b^2}\psi +\frac{q^2}{1+b^2}n_1^2\psi-\frac{iq}{\bs}n_1n_2\psi'&\\
&\qquad\quad+\frac{c_{\perp}q^2}{2a_{\perp}}\psi|^2&\\
&+\frac{a_{\|}}{q^3}|(\frac{n_1}{\bs}-\co)(-\frac{q^2}{\bs}n_1 \psi +iqn_2\psi'+q^2\co \psi)&\\
&\qquad\quad+[(\frac{iq}{\bs}n_1\psi+n_2\psi'-iq\co \psi)n_2]'|^2&\\
&+\frac{c_{\|}}{q}|\frac{iq}{\bs}n_1 \psi+n_2 \psi'-iq\co \psi|^2,&
\end{flalign*}
and $G(|\psi|^2)=g(|\psi|^2-1)^2+(|\psi|^2{}')^2+\frac{1}{q^2}(|\psi|^2{}'')^2+\frac{1}{q^6}(|\psi|^2{}''')^2.$

We take the first variation of $J(\n,\psi)$ with respect to $\n$ and $\psi$ respectively. Let $(\mathbf{u},\phi)$ be test functions where $ \mathbf{u} \in H^1(-L,L),\;\phi \in H^3(-L,L)\text{ with } \phi \in H^2_0(-L,L)$ with $\Re\{\phi''\cg\}(-L)=\Re\{\phi''\cg\}(L)=0.$

Let $\epsilon_0>0$. For every $\epsilon \in [-\epsilon_0,\epsilon_0]$, define $\n_\epsilon=\frac{\n+\epsilon\mathbf{u}}{|\n+\epsilon\mathbf{u}|}.$ Since $|\n_\epsilon|=1$, $\n_\epsilon$ belongs to the same space as $\n$ . Write
$\n_\epsilon = \n +\epsilon P(\n)\mathbf{u}+o(\epsilon)$ where $P(\n)=\mathbf{I}-\n\otimes\n$ is a projection tensor.
We have:

$$J_1(\n_\epsilon,\psi)=J_1(\n,\psi)+\epsilon\int_{-L}^L \frac{\partial E}{\partial \n}\cdot P(\n)\mathbf{u}+\frac{\partial E}{\partial \nabla\n}\cdot\nabla(P(\n)\mathbf{u})\:\:+o(\epsilon).$$

\begin{align*}
\left.\frac{d}{d\epsilon} J_1(\n_\epsilon,\psi)\right|_{\epsilon=0}={}&\int_{-L}^L \frac{\partial E}{\partial \n}\cdot  P(\n)\mathbf{u}+\frac{\partial E}{\partial \nabla\n}\cdot\nabla(P(\n)\mathbf{u})\,dx=\int_{-L}^L \frac{\partial E}{\partial \n}\cdot P(\n)\mathbf{u}\,dx.
\end{align*}
 The components of $\left.\displaystyle\frac{d}{d\epsilon} J_1(\n_\epsilon,\psi)\right|_{\epsilon=0}$ are:
\begin{align*}
& \int_{-L}^L \left\{(1-n^2_1)\left(\frac{n_1-n_1^0}{\tau}\right)-n_1n_2\left(\frac{n_2-n_2^0}{\tau}\right)-n_1n_3\left(\frac{n_3-n_3^0}{\tau}\right)\right\}u_1\,dx\\
& \int_{-L}^L \left\{-n_1n_2\left(\frac{n_1-n_1^0}{\tau}\right)+(1-n^2_2)\left(\frac{n_2-n_2^0}{\tau}\right)-n_2n_3\left(\frac{n_3-n_3^0}{\tau}\right)\right\}u_2\,dx\\
& \int_{-L}^L \left\{-n_1n_3\left(\frac{n_1-n_1^0}{\tau}\right)-n_2n_3\left(\frac{n_2-n_2^0}{\tau}\right)+(1-n^2_3)\left(\frac{n_3-n_3^0}{\tau}\right)\right\}u_3\,dx.
\end{align*}
Similarly, we have
\begin{align*}
\left.\frac{d}{d\epsilon} J_2(\n_\epsilon)\right|_{\epsilon=0}={}&\int_{-L}^L \frac{\partial  H}{\partial \nabla\n}\cdot\nabla(P(\n)\mathbf{u})\,dx=K\int_{-L}^L \nabla\n.\nabla(P(\n)\mathbf{u})\,dx&\\
={}&K\int_{-L}^L \nabla\n\cdot\nabla \mathbf{u}-|\nabla \n|^2 \n\cdot\mathbf{u}\,dx, \text{ since } |\n|=1.
\end{align*}
\noindent The components of $\left.\displaystyle\frac{d}{d\epsilon} J_2(\n_\epsilon)\right|_{\epsilon=0}$ are:
\begin{align*}
& K\int_{-L}^L -(n_1'^2+n_2'^2+n_3'^2)n_1u_1+n_1'u_1'\,dx\\
& K\int_{-L}^L -(n_1'^2+n_2'^2+n_3'^2)n_2u_2+n_2'u_2'\,dx\\
& K\int_{-L}^L -(n_1'^2+n_2'^2+n_3'^2)n_3u_3+n_3'u_3'\,dx.
\end{align*}
$\displaystyle\left.\frac{d}{d\epsilon} J_3(\n_\epsilon,\psi)\right|_{\epsilon=0}=\int_{-L}^L \frac{\partial I}{\partial \n}\cdot P(\n)\mathbf{u}\,dx$, with components
\begin{align*}
& \int_{-L}^L -\frac{PE}{\bs} |\psi|^2n_1n_3u_1\,dx\allowdisplaybreaks\\
& \int_{-L}^L -\frac{PE}{\bs} |\psi|^2n_2n_3u_2\,dx\\
& \int_{-L}^L \frac{PE}{\bs} |\psi|^2(1-n_3^2)u_3\,dx.
\end{align*}
$\displaystyle\left.\frac{d}{d\epsilon} J_4(\n_\epsilon,\psi)\right|_{\epsilon=0}=\int_{-L}^L \frac{\partial F}{\partial \n}\cdot P(\n)\mathbf{u}+\frac{\partial F}{\partial \nabla\n}\cdot\nabla(P(\n)\mathbf{u})\,dx$,  with components
\begin{align*}
&\int_{-L}^L \left((1-n^2_1)F_{n_1}-n_1n_2F_{n_2}+(1-n^2_1)'F_{n'_1}-(n_1n_2)'F_{n'_2}\right)u_1+\left((1-n^2_1)F_{n'_1}\right.&\\
&\qquad\left.-n_1n_2F_{n'_2}\right)u'_1\,dx&\\
&\int_{-L}^L \left(-n_1n_2F_{n_1}+(1-n^2_2)F_{n_2}-(n_1n_2)'F_{n'_1}+(1-n^2_2)'F_{n'_2}\right)u_2+\left(-n_1n_2F_{n'_1}\right.&\\
&\qquad\left.+(1-n^2_2)F_{n'_2}\right)u'_2\,dx&\\
&\int_{-L}^L \left(-n_1n_3F_{n_1}-n_2n_3F_{n_2}-(n_1n_3)'F_{n'_1}-(n_2n_3)'F_{n'_2}\right)u_3+\left(-n_1n_3F_{n'_1}-n_2n_3F_{n'_2}\right)u'_3\,dx.
\end{align*}
Note that we use $F_{n_1}$ to denote the derivative of $F$ with respect to $n_1$, and we calculate it using the fact that $\partial_x(|z|^2)=2 \Re(\overline{z}_xz).$ For instance,
\begin{align*}
&&F_{n_1}=2\Re&\left\{ \frac{a_\perp}{q}\left(\frac{i}{\bs}n'_2\cg+\frac{2i}{\bs}n_2\overline{\psi'}+\frac{2q}{1+b^2}n_1\cg\right)\right.\\
&&&\qquad\left(\frac{\psi''}{q}-\frac{i}{\bs}n_1'n_2\psi-\frac{i}{\bs}n_1n'_2\psi-\frac{2i}{\bs}n_1n_2\psi'-2n_2'n_2\frac{\psi'}{q}\right.\\
&&&\qquad\qquad\left.-n^2_2\frac{\psi''}{q}-\frac{q}{1+b^2}\psi +\frac{q}{1+b^2}n_1^2\psi+\frac{c_{\perp}q}{2a_{\perp}}\psi\right)\displaybreak[3]\\
&&&+\frac{a_{\|}}{q}\left(\frac{-2q}{1+b^2}n_1\cg-\frac{2i}{\bs}n_2\overline{\psi'}+\frac{2q}{\bs}\co \cg-\frac{i}{\bs}n'_2\cg\right)\\
&&& \qquad\left(\frac{-q}{1+b^2}n_1^2\psi+\frac{2i}{\bs}n_1n_2\psi'+\frac{2q}{\bs}\co n_1\psi-2i\co n_2\psi'\right.\\
&&&\qquad\qquad-q\cos^2\theta\psi+\frac{i}{\bs}n'_1n_2\psi+\frac{i}{\bs}n_1n'_2\psi+2n_2'n_2\frac{\psi'}{q}\\
&&&\qquad\qquad\left.+n_2^2\frac{\psi''}{q}-i\co n_2'\psi\right)\\
&&&\left.+qc_{\|}\left(\frac{-i}{\bs}\cg\right)\left(\frac{i}{\bs}n_1\psi+n_2\frac{\psi'}{q}-i\co \psi\right)\right\}
\end{align*}

Adding the components of the above variations, we get the first 3 Euler-Lagrange Equations in weak form, (\ref{sE-L1}), (\ref{sE-L2}) and (\ref{sE-L3}).

For the fourth Euler-Lagrange equation, we note that the energy has the general form $\displaystyle\int a_1 |g_1|^2+a_2 |g_2|^2+...+a_k|g_k|^2\,dx$ where $g_1$, $g_2$, ... and $g_k$ are linear in $\psi,\psi',\psi''$. We'll take for example $f(\psi)=\displaystyle\int a_1 |b_1\psi|^2+a_2 |b_2\psi|^2\,dx$. Then
$f(\psi+\epsilon\phi)= \displaystyle\int a_1 [b_1(\psi+\epsilon\phi)][\overline{b_1}(\cg+\epsilon\overline{\phi})]+a_2 [b_2(\psi+\epsilon\phi)][\overline{b_2}(\cg+\epsilon\overline{\phi})]\,dx,$ where $\phi$ is a test function.
\begin{align*}
\left.\frac{d}{d\epsilon}f(\psi+\epsilon\phi)\right|_{\epsilon=0}=&\int a_1 b_1 \phi \overline{b_1} \cg+a_1 b_1 \psi \overline{b_1} \overline{\phi}+a_2 b_2 \phi \overline{b_2} \cg+a_2 b_2 \psi \overline{b_2} \overline{\phi}\,dx\\
=& 2\Re\int (a_1 \overline{g_\psi}g +a_2 \overline{h_\psi}h) \overline{\phi}\,dx.
\end{align*}
Proceeding in a similar way as in deriving the first three equations, we add the different variations to get
\begin{align*}
2\Re \int_{-L}^L&\frac{1}{2}\delta_\tau \psi\overline{\phi}+F_\psi\overline{\phi}+F_{\psi'}\overline{\phi'}+F_{\psi''}\overline{\phi''}+2g(\psi\overline{\phi})(|\psi|^2-1)+2(\psi\overline{\phi})'|\psi|^2{}'&\\
&+\frac{2}{q^2}(\psi\overline{\phi})''|\psi|^2{}''+\frac{2}{q^6}(\psi\overline{\phi})'''|\psi|^2{}'''+\frac{\rho}{q^6}(\psi'''\overline{\phi'''})+\frac{PE}{\bs}n_3(\psi\overline{\phi})\,dx=0,&\nonumber
\end{align*}
Note that $F_\psi,F_{\psi'},F_{\psi''},F_{\psi'''}$ are not derivatives but notation for expressions that are sums of terms like $\overline{g_\psi} g,\overline{g_{\psi'}} g,\overline{g_{\psi''}} g, \overline{g_{\psi'''}} g$ respectively.
\subsection{Proof of $L^2$-Local Estimates (\ref{LocEst})}

We replace the test function in (\ref{ELtau}) by $\Delta_{-h}[(\Delta_hn_1)\varphi^2]$ for a small $h>0$ with $\varphi$ as in (\ref{cut-off}). We drop the superscript $\tau$ for convenience throughout this proof.
\begin{flalign*}
&\int_{\Omega_T} \left\{(1-n^2_1)\delta _\tau n_1-n_1n_2\delta _\tau n_2-n_1n_3\delta _\tau n_3+F_{n_1}-n^2_1F_{n_1}-n_1n_2F_{n_2}-2n'_1n_1F_{n'_1}\right.&\\
&{}\:\left.-n'_1n_2F_{n'_2}-n_1n'_2F_{n'_2}-Kn_1'^2n_1-Kn_2'^2n_1-Kn_3'^2n_1-\frac{PE}{\bs} |\psi|^2n_1n_3\right\}\Delta_{-h}[(\Delta_hn_1)\varphi^2]&\\
&{}\:+\left(F_{n'_1}-n^2_1F_{n'_1}-n_1n_2F_{n'_2}\right)\Delta_{-h}[(\Delta_hn'_1)\varphi^2+2(\Delta_hn_1)\varphi'\varphi]\,dxdt=0.&
\end{flalign*}
 Using the fact that $\int f\Delta_{-h}g\,dx=-\int\Delta_hfg\,dx$ and rearranging the terms,
\begin{align}
\label{eq-n}
K\int_{\Omega_T}(\Delta_hn_1'\varphi)^2&=\int_{\Omega_T}[(1-n^2_1)\delta _\tau n_1-n_1n_2\delta _\tau n_2-n_1n_3\delta _\tau n_3]\Delta_{-h}[\Delta_h n_1 \varphi^2]&\\
&\qquad-\Delta_h[F_{n_1}-n^2_1F_{n_1}-n_1n_2F_{n_2}-2n'_1n_1F_{n'_1}-n'_1n_2F_{n'_2}-n_1n'_2F_{n'_2}&\nonumber\allowdisplaybreaks\\
&\qquad\qquad\:\:-Kn_1'^2n_1-Kn_2'^2n_1-Kn_3'^2n_1-\frac{PE}{\bs} |\psi|^2n_1n_3]\Delta_h n_1 \varphi^2&\nonumber\\
&\qquad-\Delta_h[F_{n'_1}-n^2_1F_{n'_1}-n_1n_2F_{n'_2}][(\Delta_hn'_1)\varphi^2+2(\Delta_hn_1)\varphi'\varphi]\,dxdt.\nonumber
\end{align}
\noindent We want to estimate the right-hand side by a constant or a small multiple of the left-hand side. We highlight a few terms, the remaining terms are approximated in a similar fashion. Recall that $|\n|=1$ and $|\psi|\le C_1$, where $C_1$ is independent of $q$.
\begin{flalign*}
*\int_{\Omega_T} |(1-n^2_1)\delta_\tau n_1\Delta_{-h}\Delta_h n_1 \varphi^2|\,dxdt&\le \epsilon\int_{\Omega_T}|\Delta_{-h}\Delta_h n_1 \varphi|^2\,dxdt+\frac{1}{\epsilon}\int_{\Omega_T}|n_1\delta_\tau n_1\varphi|^2\,dxdt&\\
&\le \epsilon\int_{\Omega_T}|\Delta_h n_1' \varphi|^2\,dxdt+\frac{1}{\epsilon}\sum_{m=1}^M\tau ||\delta_\tau n_1||_{L^2(\Omega)}&\\
&\le \epsilon\int_{\Omega_T}|\Delta_h n_1' \varphi|^2\,dxdt+C,
\end{flalign*}
\noindent where we have used Young's inequality for conjugate H{\"o}lder exponents and the energy dissipation property (\ref{EnrgyDsp}).
\begin{flalign*}
*\int_{\Omega_T} |n_{1h}'n_1\Delta_h n_1' &\Delta_h n_1\varphi^2|\,dxdt\le \epsilon\int_{\Omega_T}|\Delta_h n_1' \varphi|^2\,dxdt+\frac{K^2}{\epsilon}\int_{\Omega_T}|n_{1h}'\Delta_h n_1\varphi|^2\,dxdt&\\
&= \epsilon\int_{\Omega_T}|\Delta_h n_1' \varphi|^2\,dxdt+\frac{K^2}{\epsilon}\int_0^T\int_{-L+\eta/2}^{L-\eta/2}|n_{1h}'\Delta_h n_1\varphi|^2\,dxdt&\\
&\le \epsilon\int_{\Omega_T}|\Delta_h n_1' \varphi|^2\,dxdt+\frac{K^2}{\epsilon}\int_0^T\sup_{(-L+\eta/2, L-\eta/2)}|\Delta_h n_1\varphi|^2\int_{-L+\eta/2}^{L-\eta-2}|n_{1h}'|^2\,dxdt&\\
&\le \epsilon\int_{\Omega_T}|\Delta_h n_1' \varphi|^2\,dxdt+\frac{K^2}{\epsilon}\int_0^TC\sup_{(-L+\eta/2, L-\eta/2)}|\Delta_h n_1\varphi|^2dt&\\
&\le \epsilon\int_{\Omega_T}|\Delta_h n_1'\varphi|^2\,dxdt+C
\end{flalign*}
\noindent where we have used the fact that for $f\in H^1(\Omega)$, $\displaystyle\sup_\Omega|f\varphi|^2\le \epsilon \int_\Omega|f'\varphi|^2\,dx+C$.

\noindent We estimate $\Delta_hF_{n_1'}$, $|\Delta_hF_{n_1'}|\le \frac{C_1}{q}\left(\frac{|\Delta_h\psi''|}{q}+...\right).$
\begin{flalign*}
* \frac{C_1}{q}\int_{\Omega_T}\left |\frac{\Delta_h\psi''}{q}\Delta_h n_1'\varphi^2\right|\,dxdt&= C_1\int_{\Omega_T}\left|\frac{\Delta_h \psi''}{q^\frac{7}{4}}\frac{\Delta_hn_1'}{q^\frac{1}{4}}\varphi\right|\,dxdt&\\
&\le \frac{C_1}{2}\int_{\Omega_T}\left|\frac{\Delta_h \psi''}{q^\frac{7}{4}}\varphi\right|^2\,dxdt+\frac{C_1}{2}\int_{\Omega_T}\left|\frac{\Delta_h n_1'}{q^\frac{1}{4}}\varphi\right|^2\,dxdt&\allowdisplaybreaks\\
&\le \frac{C_1}{q^\frac{1}{2}}\frac{1}{q^3}\int_{\Omega_T}|\Delta_h \psi''\varphi|^2\,dxdt+\frac{1}{q^\frac{1}{2}}\int_{\Omega_T}|\Delta_h n_1'\varphi|^2\,dxdt+C.
\end{flalign*}
Going back to equation $(\ref{eq-n})$, we estimate all the right-hand side terms and deduce
\begin{align}
\label{est-n}
K\int_{\Omega_T}|\Delta_hn'_1\varphi|^2&\,dxdt\le \epsilon\int_{\Omega_T}|\Delta_hn'_1\varphi|^2\,dxdt+\epsilon\int_{\Omega_T}|\Delta_hn'_2\varphi|^2\,dxdt+\epsilon\int_{\Omega_T}|\Delta_hn'_3\varphi|^2\,dxdt\\
&+\frac{\epsilon}{q^3}\int_{\Omega_T}|\Delta_h\psi''\varphi|^2\,dxdt+\frac{C_1}{q^{\frac{1}{2}}}\int_{\Omega_T}|\Delta_hn'_1\varphi|^2\,dxdt+\frac{C_1}{q}\int_{\Omega_T}|\Delta_hn'_1\varphi|^2\,dxdt\nonumber\\
&+\frac{C_1}{q}\int_{\Omega_T}|\Delta_hn'_2\varphi|^2\,dxdt+\frac{C_1}{q^{\frac{1}{2}}}\frac{1}{q^3}\int_{\Omega_T}|\Delta_h\psi''\varphi|^2\,dxdt+C,\nonumber
\end{align}
where $C_1$ is independent of $q$ and $\eta$.

For the fourth Euler-Lagrange equation, we use the test function $\phi=\Delta_{-h}[\Delta_h \psi \varphi^2].$ with $\varphi$ being the same cut-off function.
\begin{flalign*}
2\Re\int_{\Omega_T}&\delta_\tau\psi\Delta_{-h}[\Delta_h \cg \varphi^2]-\Delta_hF_\psi[\Delta_h \cg \varphi^2]-\Delta_hF_{\psi'}[\Delta_h \cg \varphi^2]'-\Delta_hF_{\psi''}[\Delta_h \cg \varphi^2]''&\\
&+2g(\psi\Delta_{-h}[\Delta_h \cg \varphi^2])(|\psi|^2-1)+2(\psi\Delta_{-h}[\Delta_h \cg \varphi^2])'|\psi|^2{}'+\frac{2}{q^2}(\psi\Delta_{-h}[\Delta_h \cg \varphi^2])''|\psi|^2{}''&\\
&+\frac{2}{q^6}(\psi\Delta_{-h}[\Delta_h \cg \varphi^2])'''|\psi|^2{}'''+\frac{\rho}{q^6}\psi'''\Delta_{-h}[\Delta_h \cg \varphi^2]'''+\frac{PE}{\bs}n_3\psi\Delta_{-h}[\Delta_h \cg \varphi^2]\,dxdt=0
\end{flalign*}
knowing that $\displaystyle\Delta_hF_{\psi''}=\left[\frac{a_{\perp}}{q^2}(1-n_2^2)^2 \frac{\Delta_h\psi''}{q}+\frac{a_{\|}}{q^2}(n_2^4)\frac{\Delta_h\psi''}{q}\right]+$ remaining terms.

\noindent We write $2\Re\{\psi\Delta_{-h}[\Delta_h\cg\varphi^2]\}=\Delta_{-h}[\Delta_h|\psi|^2\varphi^2]-\Delta_{-h}[\Delta_h\cg h\Delta_h\psi\varphi^2]-2\Delta_{-h}\cg\Delta_{-h}\psi\varphi_{-h}^2$ and isolate the terms we want to estimate.
\begin{flalign}
\label{eq-psi}
 2\int_{\Omega_T}&\left[\frac{a_{\perp}}{q^2}(1-n_2^2)^2 \frac{|\Delta_h\psi''|^2}{q}+\frac{a_{\|}}{q^2}(n_2^4)\frac{|\Delta_h\psi''|^2}{q}\right]\varphi^2+\frac{1}{q^6}|\Delta_h|\psi|^2{}'''\varphi|^2+\frac{\rho}{q^6}|\Delta_h\psi'''\varphi|^2\,dxdt&\\
=2\Re&\int_{\Omega_T}\delta_\tau\psi\Delta_{-h}[\Delta_h \cg \varphi^2]-\Delta_hF_\psi[\Delta_h \cg \varphi^2]-\Delta_hF_{\psi'}[\Delta_h \cg \varphi^2]'+\left\{-\Delta_hF_{\psi''}[\Delta_h \cg \varphi^2]''\right.\nonumber&\\
&+\left[\frac{2a_{\perp}}{q^2}(1-n_2^2)^2 \frac{|\Delta_h\psi''|^2}{q}+\frac{2a_{\|}}{q^2}(n_2^4)\frac{|\Delta_h\psi''|^2}{q}\right]\varphi^2\left.\right\}+2g(\psi\Delta_{-h}[\Delta_h \cg \varphi^2])(|\psi|^2-1)&\nonumber\\
&+2(\psi\Delta_{-h}[\Delta_h \cg \varphi^2])'|\psi|^2{}'+\frac{2}{q^6}(\psi\Delta_{-h}[\Delta_h \cg \varphi^2])''|\psi|^2{}''+\left\{\right.\frac{2}{q^6}(\psi\Delta_{-h}[\Delta_h \cg \varphi^2])'''|\psi|^2{}'''&\nonumber\\
&+\frac{1}{q^6}|\Delta_h|\psi|^2{}'''\varphi|^2\left.\right\}+\left\{\right.\frac{\rho}{q^6}\psi'''\Delta_{-h}[\Delta_h \cg \varphi^2]'''+\frac{\rho}{q^6}|\Delta_h\psi'''\varphi|^2\left.\right\}+\frac{PE}{\bs}n_3\psi\Delta_{-h}[\Delta_h \cg \varphi^2]\,dxdt.\nonumber
\end{flalign}

\noindent As before, we want to estimate the right-hand side by a constant or a small multiple of the left-hand side. We highlight only a few terms. Recall that $\dfrac{|\psi'|}{q}\le C_1$.

\noindent One term we consider,
\begin{flalign*}
\int_{\Omega_T}2\Re\{-\Delta_hF_\psi\Delta_h\cg\varphi^2\}\,dxdt&\le 2\int_0^T \sup\limits_{(-L,L)}|\Delta_h\psi\varphi|\int_{-L}^L|\Delta_hF_\psi\varphi|\,dxdt&\\
&\le C_1\int_0^T \sup\limits_{(-L,L)}|\psi'|\int_{-L}^L|\Delta_hF_\psi\varphi|\,dxdt&\\
&\le C_1q\int_{\Omega_T}|\Delta_hF_\psi\varphi|\,dxdt,
\end{flalign*}
where $|\displaystyle\Delta_hF_\psi|\le \frac{C_1}{q}(|n_1'|+...)(|\frac{\Delta_h\psi''}{q}|+...)$.
\begin{flalign*}
* \,C_1\int_{\Omega_T}|\frac{\Delta_h\psi''}{q}n_1'\varphi|\,dxdt& \le \frac{\epsilon}{q}\int_{\Omega_T}|\frac{\Delta_h\psi''\varphi}{q}|^2\,dxdt+\frac{C_1^2q}{\epsilon}\int_{\Omega_T}|n_1'|^2\,dxdt&\\
&\le \frac{\epsilon}{q^3}\int_{\Omega_T}|\Delta_h\psi''\varphi|^2\,dxdt+C.&
\end{flalign*}
Another term we consider,
\begin{flalign*}
&\int_{\Omega_T}2\Re\left\{-\Delta_hF_{\psi''}[\Delta_h \cg \varphi^2]''+\left[\frac{2a_{\perp}}{q^2}(1-n_2^2)^2 \frac{|\Delta_h\psi''|^2}{q}+\frac{2a_{\|}}{q^2}(n_2^4)\frac{|\Delta_h\psi''|^2}{q}\right]\varphi^2\right\}\,dxdt,&
\end{flalign*}
where $\left|-\Delta_hF_{\psi''}+\left[\frac{a_{\perp}}{q^2}(1-n_2^2)^2 \frac{\Delta_h\psi''}{q}+\frac{a_{\|}}{q^2}(n_2^4)\frac{\Delta_h\psi''}{q}\right]\right|\le\frac{C_1}{q^2}(|n_{2h}\psi_h\Delta_hn'_1|+$ remaining terms.
\begin{flalign*}
*\frac{C_1}{q^2}\int_{\Omega_T}|n_{2h}\psi_h\Delta_hn'_1\Delta_h\overline{\psi''}\varphi^2|\,dxdt&\le \frac{C_1}{q^2}\int_{\Omega_T}|\Delta_hn'_1\Delta_h\psi''\varphi^2|\,dxdt&\\
&\le C_1q^{\frac{1}{2}}\int_{\Omega_T}|\frac{\Delta_hn'_1}{q^{\frac{1}{2}}}\varphi|^2\,dxdt+C_1\frac{1}{q^{\frac{1}{2}}}\int_{\Omega_T}|\frac{\Delta_h\overline{\psi''}}{q^{\frac{3}{2}}}\varphi|^2\,dxdt&\\
&=                                                                                                                                                                                                                                                                                 \frac{C_1}{q^{\frac{1}{2}}}\int_{\Omega_T}|\Delta_hn'_1\varphi|^2\,dxdt+\frac{C_1}{q^{\frac{1}{2}}} \frac{1}{q^3}\int_{\Omega_T}|\Delta_h\psi''\varphi|^2\,dxdt.&
\end{flalign*}
A third term we consider,
\begin{flalign*}
&\frac{2}{q^6}\int_{\Omega_T}2\Re\left\{(\psi\Delta_{-h}[\Delta_h\cg\varphi^2])'''|\psi|^2{}'''\right\}\,dxdt\\
&=\frac{2}{q^6}\int_{\Omega_T}(\Delta_{-h}[\Delta_h|\psi|^2\varphi^2])'''|\psi|^2{}'''\,dxdt-\frac{2}{q^6}\int_{\Omega_T}(\Delta_{-h}[\Delta_h\psi h\Delta_h\cg\varphi^2])'''|\psi|^2{}'''\,dxdt&\allowdisplaybreaks\\
&\quad-\frac{4}{q^6}\int_{\Omega_T}(\Delta_{-h}\psi\Delta_{-h}\cg\varphi_{-h}^2)'''|\psi|^2{}'''\,dxdt&\\
&=-\frac{2}{q^6}\int_{\Omega_T}(\Delta_h|\psi|^2\varphi^2)'''\Delta_h|\psi|^2{}'''\,dxdt+\frac{2}{q^6}\int_{\Omega_T}(\Delta_h\psi h\Delta_h\cg\varphi^2)'''\Delta_h|\psi|^2{}'''\,dxdt&\\
&\quad-\frac{4}{q^6}\int_{\Omega_T}(\Delta_{-h}\psi\Delta_{-h}\cg\varphi_{-h}^2)'''|\psi|^2{}'''\,dxdt&
\end{flalign*}
\begin{flalign*}
*&\frac{2h}{q^6}\int_{\Omega_T}|\Delta_h\psi|^2{}'''\varphi^2\Delta_h|\psi|^2{}'''\,dxdt&\\
&=\frac{2h}{q^6}\int_{\Omega_T}2\Re\{\Delta_h\psi'''\Delta_h\cg\}\varphi^2\Delta_h|\psi|^2{}'''\,dxdt+ \frac{2h}{q^6}\int_{\Omega_T}2\Re\{3\Delta_h\psi''\Delta_h\overline{\psi'}\}\varphi^2\Delta_h|\psi|^2{}'''\,dxdt&\\
&\le \frac{\epsilon_1}{q^6}\int_{\Omega_T}(\Delta_h|\psi|^2{}'''\varphi)^2\,dxdt+\frac{4^2h^2}{\epsilon_1q^4}\int_{\Omega_T}|\Delta_h\psi'''\frac{\Delta_h\cg}{q}\varphi|^2\,dxdt&\\
&\quad+ \frac{\epsilon_1}{q^6}\int_{\Omega_T}(\Delta_h|\psi|^2{}'''\varphi)^2\,dxdt+\frac{12^2h^2}{\epsilon_1q^6}\int_{\Omega_T}|\Delta_h\psi''\Delta_h\overline{\psi'}\varphi|^2\,dxdt&\\
&\le\frac{\epsilon_1}{q^6}\int_{\Omega_T}(\Delta_h|\psi|^2{}'''\varphi)^2\,dxdt+\frac{4^2C_1h^2}{\epsilon_1q^4} \int_{\Omega_T}|\Delta_h\psi'''\varphi|^2\,dxdt&\\
&\quad+\frac{12^2}{\epsilon_1q^6}\int_0^T\sup_{(-L,L)}|\Delta_h\psi'\varphi|^2\int_{-L}^L|\psi''(x+h)-\psi''(x)|^2\,dxdt\qquad\left(\text{ choose } h^2<\frac{\epsilon\epsilon_1\rho}{4^2C_1q^2}\right)&\\
&\le\frac{\epsilon}{q^6}\int_{\Omega_T}(\Delta_h|\psi|^2{}'''\varphi)^2\,dxdt+\frac{\rho}{q^6} \int_{\Omega_T}|\Delta_h\psi'''\varphi|^2\,dxdt+\frac{\epsilon}{q^3}\int_{\Omega_T}|\Delta_h\psi''\varphi|^2\,dxdt+C.&
\end{flalign*}
Note that the constant on the right-hand side is dependent on $\rho$.

Going back to equation $(\ref{eq-psi})$, we estimate all the terms and deduce
\begin{flalign*}
 & 2\int_{\Omega_T}\frac{a_{\perp}}{q^2}(1-n_2^2)^2 \frac{|\Delta_h\psi''|^2}{q}+\frac{a_{\|}}{q^2}(n_2^4)\frac{|\Delta_h\psi''|^2}{q}\varphi^2+\frac{1}{q^6}|\Delta_h|\psi|^2{}'''\varphi|^2+\frac{\rho}{q^6}|\Delta_h\psi'''\varphi|^2\,dxdt&\\
 &\le \epsilon\int_{\Omega_T}|\Delta_h\n'\varphi|^2\,dxdt+\frac{\epsilon}{q^3}\int_{\Omega_T}|\Delta_h\psi''\varphi|^2\,dxdt+\frac{C_1}{q^{\frac{1}{2}}}\int_{\Omega_T}|\Delta_h\n'\varphi|^2\,dxdt&\\
 &+\frac{C_1}{q^{\frac{1}{2}}}\frac{1}{q^3}\int_{\Omega_T}|\Delta_h\psi''\varphi|^2\,dxdt+\frac{\epsilon}{q^6}\int_{\Omega_T}(\Delta_h|\psi|^2{}'''\varphi)^2\,dxdt+\frac{\epsilon\rho}{q^6}\int_{\Omega_T}|\Delta_h\psi'''\varphi|^2\,dxdt+C(\rho).&
\end{flalign*}
However,
\begin{flalign*}
\frac{2a_{\perp} a_{\|}}{a_{\perp}+a_{\|}}\frac{1}{q^3}\int_{\Omega_T}|\Delta_h\psi''\varphi|^2\,dxdt\le&2\int_{\Omega_T}[a_{\perp}(1-n_2^2)^2+a_{\|}(n_2^4)]\frac{1}{q^3}|\Delta_h\psi''\varphi|^2\,dxdt&\\
=&2\int_{\Omega_T}\Re\left\{\left[\frac{a_{\perp}}{q^2}(1-n_2^2)^2 \frac{\Delta_h\psi''}{q}+\frac{a_{\|}}{q^2}(n_2^4)\frac{\Delta_h\psi''}{q}\right]\Delta_h\overline{\psi''}\varphi^2\right\}\,dxdt.&
\end{flalign*}
Therefore,
\begin{align}
\label{est-psi}
 &\frac{2a_{\perp} a_{\|}}{a_{\perp}+a_{\|}}\frac{1}{q^3}\int_{\Omega_T}|\Delta_h\psi''\varphi|^2\,dxdt+\frac{2}{q^6}\int_{\Omega_T}(\Delta_h|\psi|^2{}'''\varphi)^2+\frac{2\rho}{q^6}\int_{\Omega_T}|\Delta_h\psi'''\varphi|^2\,dxdt&\allowdisplaybreaks\\
 &\le \epsilon\int_{\Omega_T}|\Delta_h\n'\varphi|^2\,dxdt+\frac{\epsilon}{q^3}\int_{\Omega_T}|\Delta_h\psi''\varphi|^2\,dxdt+\frac{C_1}{q^{\frac{1}{2}}}\int_{\Omega_T}|\Delta_h\n'\varphi|^2\,dxdt\nonumber&\\
 &+\frac{C_1}{q^{\frac{1}{2}}}\frac{1}{q^3}\int_{\Omega_T}|\Delta_h\psi''\varphi|^2\,dxdt+\frac{\epsilon}{q^6}\int_{\Omega_T}(\Delta_h|\psi|^2{}'''\varphi)^2\,dxdt+\frac{\epsilon\rho}{q^6}\int_{\Omega_T}|\Delta_h\psi'''\varphi|^2\,dxdt+C(\rho) \nonumber
\end{align}
We add the two estimates $(\ref{est-n})$ and $(\ref{est-psi})$, along with the corresponding ones for $n_2$ and $n_3$,
\begin{flalign*}
&\left(K-\epsilon-\frac{C_1}{q}-\frac{C_1}{q^{\frac{1}{2}}}\right)\int_{\Omega_T}|\Delta_h\n'\varphi|^2\,dxdt+\left(\frac{2a_{\perp} a_{\|}}{a_{\perp}+a_{\|}}-\epsilon-\frac{C_1}{q^{\frac{1}{2}}}\right)\frac{1}{q^3}\int_{\Omega_T}|\Delta_h\psi''\varphi|^2\,dxdt&\allowdisplaybreaks\\
&+\left(2-\epsilon\right)\frac{1}{q^6}\int_{\Omega_T}(\Delta_h|\psi|^2{}'''\varphi)^2\,dxdt+\left(2-\epsilon\right)\frac{\rho}{q^6}\int_{\Omega_T}|\Delta_h\psi'''\varphi|^2\,dxdt\le C(\rho).&
\end{flalign*}
Choosing $\epsilon$ sufficiently small, and for $q$ sufficiently large, and since $\varphi=1$ on $(-L+\eta,L-\eta)$,
\begin{equation*}
\int_0^T\int_{-L+\eta}^{L-\eta}|\Delta_h\n'|^2+\frac{1}{q^3}|\Delta_h\psi''|^2+\frac{1}{q^6}|\Delta_h|\psi|^2{}'''|^2+\frac{\rho}{q^6}|\Delta_h\psi'''|^2\,dxdt\le C(\rho).
\end{equation*}
Letting $h\to 0$, we deduce that $n''_1,\, n''_2,\, n''_3,\, \psi''',\,|\psi|^2{}^{(4)},\,\psi^{(4)} \in L^2((-L+\eta,L-\eta)\times[0,T]).$ To remove the dependence on $\rho$, so we repeat all the estimates replacing difference quotients by derivatives. For instance,
\begin{flalign*}
&-\frac{2}{q^6}\int_{\Omega_T}[\psi(\overline{\psi'}\varphi^2)'+\overline{\psi}(\psi'\varphi^2)']''|\psi|^2{}^{(4)}\,dxdt&\\
&=-\frac{2}{q^6}\int_{\Omega_T}[(\psi\overline{\psi''}+\overline{\psi}\psi'')\varphi^2]''|\psi|^2{}^{(4)}\,dxdt-\frac{2}{q^6}\int_{\Omega_T}[(\psi\overline{\psi'}+\overline{\psi}\psi')2\varphi'\varphi]''|\psi|^2{}^{(4)}\,dxdt&\\
&=-\frac{2}{q^6}\int_{\Omega_T}[(|\psi|^2{}''-2\overline{\psi'}\psi')\varphi^2]''|\psi|^2{}^{(4)}\,dxdt-\frac{2}{q^6}\int_{\Omega_T}[(|\psi|^2{}')2\varphi'\varphi]''|\psi|^2{}^{(4)}\,dxdt.&
\end{flalign*}
\begin{flalign*}
*\frac{4}{q^6}\int_{\Omega_T}|\psi'''\overline{\psi'}\varphi^2|\psi|^2{}^{(4)}|\,dxdt&\le\frac{\epsilon}{q^6}\int_{\Omega_T}(|\psi|^2{}^{(4)}\varphi)^2\,dxdt+\frac{4^2}{\epsilon q^4}\int_{\Omega_T}|\psi'''\frac{\overline{\psi'}}{q}\varphi|^2\,dxdt&\\
&\le\frac{\epsilon}{q^6}\int_{\Omega_T}(|\psi|^2{}^{(4)}\varphi)^2\,dxdt+\frac{C_1}{q}\frac{1}{q^3}\int_{\Omega_T}|\psi'''\varphi|^2\,dxdt
\end{flalign*}
We proceed as before and arrive to
\begin{flalign*}
&\left(K-\epsilon-\frac{C_1}{q}-\frac{C_1}{q^{\frac{1}{2}}}\right)\int_{\Omega_T}|\n''\varphi|^2\,dxdt+\left(\frac{2a_{\perp} a_{\|}}{a_{\perp}+a_{\|}}-\epsilon-\frac{C_1}{q^{\frac{1}{2}}}-\frac{C_1}{q}\right)\frac{1}{q^3}\int_{\Omega_T}|\psi'''\varphi|^2\,dxdt&\\
&+\left(2-\epsilon\right)\frac{1}{q^6}\int_{\Omega_T}(|\psi|^{(4)}\varphi)^2\,dxdt+\left(2-\epsilon\right)\frac{\rho}{q^6}\int_{\Omega_T}|\psi^{(4)}\varphi|^2\,dxdt\le C,&
\end{flalign*}
which leads to  (\ref{LocEst}).
\subsection{Time-Quotient Reasoning in Higher Estimates}
To prove higher local regularity, we repeat a similar process as in Appendix B. Here we just highlight where the extra condition, that $\displaystyle\int_{\Omega'}|\n^0{}''|^2\,dx$ is initially bounded, arises from. For that we consider only one piece of the Euler-Lagrange equation, after replacing the test function by $\mathbf{u}=\Delta_{-h}[\Delta_h\n'\varphi^2]'$.
\begin{flalign*}
 &\int_{\Omega_T}[\delta_\tau\n-(\delta_\tau\n\cdot\n)\n]\Delta_{-h}[\Delta_h\n'\varphi^2]'\,dxdt=-\int_{\Omega_T}\Delta_h[\delta_\tau\n-(\delta_\tau\n\cdot\n)\n][\Delta_h\n'\varphi^2]'\,dxdt&\\
&\qquad=\int_{\Omega_T}\Delta_h[\delta_\tau\n'][\Delta_h\n'\varphi^2]\,dxdt+\int_{\Omega_T}\Delta_h[(\delta_\tau\n\cdot\n)\n][\Delta_h\n'\varphi^2]'\,dxdt:=\rm I+\rm II.
\end{flalign*}
Using the fact that
\begin{flalign*}
(\delta_\tau\omega\cdot\omega)_m&=\frac{1}{\tau}(\omega^m-\omega^{m-1})\cdot\frac{1}{2}[(\omega^m-\omega^{m-1})+(\omega^m+\omega^{m-1})]&\allowdisplaybreaks\\
&=\frac{1}{2\tau}|\omega^m-\omega^{m-1}|^2+\frac{1}{2\tau}(|\omega^m|^2-|\omega^{m-1}|^2),
\end{flalign*}
we write
\begin{flalign*}
\rm I&=\int_{\Omega_T}\delta_\tau(\Delta_h\n')\Delta_h\n'\varphi^2\,dxdt&\\
&=\int_{\Omega_T}\frac{\tau}{2}|\delta_\tau\Delta_h\n'|^2\varphi^2\,dxdt+\int_{\Omega_T}\frac{1}{2\tau}|\Delta_h\n^M{}'|^2\varphi^2\,dxdt-\int_{\Omega_T}\frac{1}{2\tau}|\Delta_h\n^0{}'|^2\varphi^2\,dxdt&\\
&=\int_{\Omega_T}\frac{\tau}{2}|\delta_\tau\Delta_h\n'|^2\varphi^2\,dxdt+\frac{1}{2}\int_\Omega|\Delta_h\n^M{}'|^2\varphi^2\,dx-\frac{1}{2}\int_\Omega|\Delta_h\n^0{}'|^2\varphi^2\,dx.
\end{flalign*}
Using the fact that
\begin{flalign*}
(\delta_\tau\n\cdot\n)_m&=\frac{1}{\tau}(\n^m-\n^{m-1})\cdot\frac{1}{2}[(\n^m-\n^{m-1})+(\n^m+\n^{m-1})]&\\
&=\frac{1}{2\tau}|\n^m-\n^{m-1}|^2+\frac{1}{2\tau}(|\n^m|^2-|\n^{m-1}|^2)\qquad\text{ with }|\n^m|=|\n^{m-1}|=1&\allowdisplaybreaks\\
&=\frac{1}{2\tau}|\n^m-\n^{m-1}|^2,
\end{flalign*}
we write
\begin{flalign*}
\rm{ II}&=\int_{\Omega_T}\Delta_h[(\delta_\tau\n\cdot\n)\n][\Delta_h\n'\varphi^2]'\,dxdt&\\
&=\int_{\Omega_T}\Delta_h(\delta_\tau\n\cdot\n)\n_h[\Delta_h\n'\varphi^2]'\,dxdt+\int_{\Omega_T}(\delta_\tau\n\cdot\n)\Delta_h\n[\Delta_h\n'\varphi^2]'\,dxdt&\\
&=\int_{\Omega_T}\Delta_h\left(\frac{\tau}{2}|\delta_\tau\n|^2\right)\n_h[\Delta_h\n'\varphi^2]'\,dxdt+\int_{\Omega_T}(\delta_\tau\n\cdot\n)\Delta_h\n[\Delta_h\n'\varphi^2]'\,dxdt.
\end{flalign*}
We can prove that
\begin{flalign*}
|{\rm II}|\le \epsilon\int_{\Omega_T}|\Delta_h\n''\varphi|^2\,dxdt+\epsilon\int_{\Omega_T}\tau|\Delta_h\delta_\tau\n'\varphi|^2\,dxdt+\epsilon \sup_{(0,T)}\int_\Omega|\Delta_h\n'\varphi|^2\,dx+C,
\end{flalign*}
to get the following estimate
\begin{flalign}
\label{Ineq}
&\left|{\rm I}+{\rm II}-\int_{\Omega_T}\frac{\tau}{2}|\delta_\tau\Delta_h\n'\varphi|^2\,dxdt-\frac{1}{2}\int_\Omega|\Delta_h\n^M{}'\varphi|^2\,dx\right|&
\\
&\le \frac{1}{2}\int_\Omega|\Delta_h\n^0{}'\varphi|^2\,dx+\epsilon\int_{\Omega_T}|\Delta_h\n''\varphi|^2\,dxdt+\epsilon\int_{\Omega_T}\tau|\Delta_h\delta_\tau\n'\varphi|^2\,dxdt&\nonumber\\
&\quad+\epsilon \sup_{(0,T)}\int_\Omega|\Delta_h\n'\varphi|^2\,dx+C.\nonumber
\end{flalign}
As can be seen from (\ref{Ineq}), in order to get an estimate using the Euler-Lagrange equations with the terms $\int_{\Omega_T}\frac{\tau}{2}|\delta_\tau\Delta_h\n'\varphi|^2\,dxdt$ and $\frac{1}{2}\int_\Omega|\Delta_h\n^M{}'\varphi|^2\,dx$, among others, on the left-hand side, we need $\frac{1}{2}\int_\Omega|\Delta_h\n^0{}'\varphi|^2\,dx$ to be initially bounded.
\subsection{Recovering Euler-Lagrange equations when $\rho\to 0$}
Recall the energy functional
$$\qquad J^\rho(\n,\psi)=\displaystyle\int_{-L}^L \left\{\frac{|\n-\n^{k-1}|^2}{2\tau}+\frac{|\psi-\psi^{k-1}|^2}{2\tau}\right\}\,dx+\mathcal{F}_q(\n,\psi;\rho)$$
\noindent where we consider $(\n,\psi)=(\n^{k,\rho},\psi^{k,\rho})$ that minimizes $J^\rho$. We want to let $\rho\to 0$, assuming that the time step $k$ is fixed. From the interior estimates (\ref{LocEst}) we have that $\n^{k,\rho}\to \n^{k,0}$ in $H^1(-L',L';\mathbb{S}^2)$ and $\psi^{k,\rho}\to\psi^{k,0}$ in $H^2(-L',L')$ for each $L'<L$. We want to prove that $\n^{k,\rho}\to \n^{k,0}$ in $H^1(-L,L;\mathbb{S}^2)$ and $\psi^{k,\rho}\to\psi^{k,0}$ in $H^2(-L,L)$.

From the lower semi-continuity of the integrals with respect to these sequences we have
$$J^0(\n^{k,0},\psi^{k,0})\leq\underset{\overline{\rho\to 0}}\lim J^\rho(\n^{k,\rho},\psi^{k,\rho}).$$

We first show that this is in fact an equality by constructing test functions $(\n_\varepsilon,\psi_\varepsilon)$. Fix $r>0$ so that $2>|\psi^{k,0}|>1/2$ on $(-L,-L+r)\cup(L-r,L)$. Set
\begin{eqnarray*}
&&\n_\varepsilon=\n^{k,0} \text{ on } (-L,L)\\
&&\psi_\varepsilon =\psi^{k,0} \text{ on } (-L+r,L-r).
\end{eqnarray*}
On $(-L,-L+r)\cup(L-r,L)$ we write $\psi^{k,0}(x)=|\psi^{k,0}|e^{i\Theta(x)}$. Since $\Theta(x)\in H^3(L-r,L')\cap H^2(-L+r,L)$, we can find $\Theta_\varepsilon\in H^3(L-r,L)$ so that $\Theta_\varepsilon(L-r)=\Theta(L-r), \Theta_\varepsilon'(L-r)=\Theta'(L-r), \Theta_\varepsilon''(L-r)=\Theta''(L-r), \Theta_\varepsilon(L)=\Theta(L), \Theta_\varepsilon'(L)=\Theta'(L)$ and so that $\Theta_\varepsilon(x)\to\Theta(x)$ in $H^2(L-r,L)$. We carry out the corresponding construction on $(-L,-L+r)$. We set $\psi_\varepsilon(x)=|\psi^{k,0}(x)|e^{i\Theta_\varepsilon(x)}$ on $(-L,-L+r)\cup(L-r,L)$.

We find that $\psi_\varepsilon(x)\in H^3(-L,L)$, $\psi_\varepsilon\to\psi^{k,0}$ in $H^2(-L,L)$ as $\varepsilon\to 0$ with  $\psi_\varepsilon$ having the correct boundary conditions at $x_2=\pm L$. This renders $(\n_\varepsilon,\psi_\varepsilon)$ a comparison function for each $\varepsilon,\rho>0$ for $J^\rho$.

We now choose, for each $\varepsilon>0$, a $\rho(\varepsilon)>0$ and small so that
$$\rho(\varepsilon)\int_{-L}^{L}|\psi_\varepsilon'''|^2\,dx<\varepsilon.$$
Then we have
$$J^0(\n^{k,0},\psi^{k,0})\leq\underset{\overline{\rho\to 0}}\lim J^\rho(\n^{k,\rho},\psi^{k,\rho})\leq\underset{\overline{\varepsilon\to 0}}\lim J^{\rho(\varepsilon)}(\n_\varepsilon,\psi_\varepsilon)=J^0(\n^{k,0},\psi^{k,0}).$$
Thus
$$J^0(\n^{k,0},\psi^{k,0})=\underset{\overline{\rho\to 0}}\lim J^\rho(\n^{k,\rho},\psi^{k,\rho}).$$
It follows that each of the integrals making up $J^\rho$ converge to their counterpart  in $J^0$. Therefore $\int_{-L}^{L}|(\n^{k,\rho})'|^2\,dx\to\int_{-L}^{L}|(\n^{k,0})'|^2\,dx$ which implies that $\n^{k,\rho}\to\n^{k,0}$ in $H^1(-L,L)$.

Expanding the first two terms in $\mathcal{F}_q$ out we get
\begin{flalign*}
&\left\{\frac{a_\perp}{q^3}\left(1-({n_2^{k,0})}^2\right)^2+\frac{a_\|}{q^3}(n_2^{k,0})^4\right\}|(\psi^{k,\rho})''|^2&\\
&+
\left\{\frac{a_\perp}{q^3}\left[\left(1-({n_2^{k,\rho})}^2\right)^2-\left(1-({n_2^{k,0})}^2\right)^2\right]+\frac{a_\|}{q^3}\left[({n_2^{k,\rho})}^4-(n_2^{k,0})^4\right]\right\}|(\psi^{k,\rho})''|^2+\cdots&\\
&=\left\{\frac{a_\perp}{q^3}\left(1-({n_2^{k,0})}^2\right)^2+\frac{a_\|}{q^3}(n_2^{k,0})^4\right\}|(\psi^{k,\rho})''|^2+{\rm I}(\rho)+{\rm II}(\rho)
\end{flalign*}
where $\int_{-L}^L {\rm I}(\rho)\,dx\to 0$ and $\int_{-L}^L {\rm II}(\rho)\,dx\to \int_{-L}^L {\rm II}(0)\,dx$ as $\rho\to 0$. Here we are using the properties that ${\rm II}(\rho)$ is at most linear in $(\psi^{k,\rho})''$ and $\n^{k,\rho}\to\n^{k,0}$ in $H^1(-L,L)$. Since the integral of this expression converges to the corresponding integral in $J^0$ we get
\begin{flalign*}
\int_{-L}^L&\left\{\frac{a_\perp}{q^3}\left(1-({n_2^{k,0})}^2\right)^2+\frac{a_\|}{q^3}(n_2^{k,0})^4\right\}|(\psi^{k,\rho})''|^2\,dx&\\
&\longrightarrow \int_{-L}^L\left\{\frac{a_\perp}{q^3}\left(1-({n_2^{k,0})}^2\right)^2+\frac{a_\|}{q^3}(n_2^{k,0})^4\right\}|(\psi^{k,0})''|^2\,dx.
\end{flalign*}
\bigskip
Due to this, $\psi^{k,\rho}\to\psi^{k,0}$ in $H^2(-L,L)$.

Finally using the strong convergence for $\psi^{k,\rho}$ and $\n^{k,\rho}$ we can show that one recovers the Euler-Lagrange equations  (\ref{sE-L1}), (\ref{sE-L2}), (\ref{sE-L3}) with $(u_1(x,k\tau),u_2(x,k\tau),$

\noindent $u_3(x,k\tau))\in H^1(-L,L)$ upon letting $\rho\to 0$.
\end{appendix}

\bibliographystyle{amsplain}

\begin{thebibliography}{26}
\bibitem{} Rieker, T.P.,  Clark, N.A.,  Smith, G.S.,  Parmar, D.S.,   Sirota, E.B., and Safinya, C.R., (1987) Phys. Rev. Lett. 59, 2658.

\bibitem{}  Ouchi, Y.,  Lee, J.,   Takezoe, H.,  Fukuda, A.,  Kondo, K., Kitamura, T. and  Mukoh, A., (1988) Jpn. J. Appl. Phys. 27, L725.

\bibitem{}  Cagnon, M., and  Durand, G., (1993) Phys. Rev. Lett. 70, 2742.

\bibitem{}  Clark, N.A.,  and  Rieker, T.P., (1988) Phys. Rev. A 37, 1053.

\bibitem{}  Mottram, N.J.,  Islam, N.U., and  Elston, S.J., (1999) Phys. Rev. E 60,
613.

\bibitem{}  Maclennann, J.E., Clarck, N.A., Handschy, M.A., and Meadows, M.R., (1990) Liq. Cryst., 7, 753.

\bibitem{} Maclennann, J.E., Handschy, M.A., and Clark, N.A., (1990) Liq. Cryst., 7, 787.

\bibitem{} Ulrich, D.C., (1995) PhD Thesis, Oxford University, U.K.

\bibitem{}  Brown, C.V., Dunn, P.E., and Jones, J.C., (1997) Eur. J. Appl. Math., 8, 281.

\bibitem{}  Nakagawa, M., and Akahane, T., (1986) J. Phys. Soc. Jpn. 55, 1516.

\bibitem{}  Nakagawa, M., (1990) Displays 11, 67.

\bibitem{} Sabater, J.,  Pena, J.M.S., and  Ot\'{o}n, J.M., (1995) J. Appl. Phys. 77, 3023.

\bibitem{}  De Meyere, A.,  Pauwels, H., and  De Ley, E., (1993) Liq. Cryst. 14, 1269.

\bibitem{} De Meyere, A., and  Dahl, I., (1994) Liq. Cryst. 17, 379.

\bibitem{}  Limat, L., (1995) J. Phys. II 5, 803.

\bibitem{} Kralj, S., Sluckin, T., (1994) Phys. Rev. E, 50, 2940.

\bibitem{} Vaupoti\u{c}., N., Kralj, S., Copi\u{c} , M., and Sluckin, T.J., (1996) Phys. Rev. E, 54, 3783.

\bibitem{} Vaupoti\u{c}., N., and Copi\u{c} , M., (2000) Phys. Rev. E, 62, 2317.

\bibitem{} Hazelwood, L.D., Sluckin, T. J., (2004) Liq. Cryst. 31, 683.

\bibitem{} Cheng, L.Z., Phillips, D., (2015) SIAM J. Appl. Math., 75, 164.

\bibitem{}  Willis, P.C.,  Clark, N.A., and  Safinya, C.R., (1992) Liq. Cryst. 11, 581.

\bibitem{} Chen, J.,  Lubensky, T., (1976) Phys. Rev. A, 14, 1202.

\bibitem{} Rothe, E., (1930) Math. Ann., 102, 650.

\bibitem{} Nirenberg, L., (1959) Ann. Scuola Norm. Sci., 13, 115.

\bibitem{} Shalaginov, A. Hazelwood, L. and Sluckin, T., (1998) Phys. Rev. E, 58, 7455.

\bibitem{} Cheng, L.Z., (2012) PhD thesis, Purdue University.

\end{thebibliography}

\end{document}